\pdfoutput=1
\documentclass{article}
\usepackage{pack}
\usepackage[backend = bibtex, sorting = none]{biblatex}
\usepackage{lineno}

\title{Computing character tables and Cartan matrices of finite monoids with fixed point counting.}

\author{Balthazar Charles\footnote{LISN, Université Paris-Saclay, 6 Rue Noetzlin, Gif-sur-Yvette, 91190, France}}
\date{}

\begin{document}
	
		
		\maketitle
						
		\begin{abstract}
			In this paper we present an algorithm for efficiently counting fixed points in a finite monoid $M$ under a conjugacy-like action. We then prove a formula for the character table of $M$ in terms of fixed points and radical, which allows for the effective computation of the character table of $M$ over a field of null characteristic, as well as its Cartan matrix, using a formula from [Thiéry '12], again in terms of fixed points. We discuss the implementation details of the resulting algorithms and provide benchmarks of their performances.
		\end{abstract}
		
	
	\section*{Introduction}
	
	The last two decades have seen the development of a new dynamic around the study of monoid representation theory. This is due to applications to certain types of discrete Markov chains and specially Markov chains used to randomly generate combinatorial objects first uncovered in the seminal article of Brown \cite{Brown.2000}.
This has lead to an exploration of the combinatorial properties of monoid representations, for instance in \cite{AyyerSchillingSteinbergThiery.2014.RTrivivalMarkovChains}, \cite{AyyerSchillingThiery.2014.SpectralGap} or \cite{Thiery.CartanMatrixMonoid}.

In this last article \cite{Thiery.CartanMatrixMonoid}, Thiéry gives a formula for the Cartan matrix of a finite monoid of $M$ in terms of number of fixed points and the character table of $M$. More precisely, the formula involves computing the cardinality of the set $\{s \in M \sepp hsk = s\}$ for any $h, k \in M$. In this paper, we set out to use this formula to effectively compute the Cartan matrix of the algebra of $M$ over a perfect field $\kbf$ of null characteristic.

Two difficulties have to be overcome in the pursuit of this goal. Firstly, the cardinality of many interesting families of monoids tends the increase very quickly. For instance, the cardinality of the \emph{full transformation monoid} $T_n$ of all functions from $\intint{1, n}$ to $\intint{1, n}$ is $n^n$, making the naive computation of $|\{s \in M \sepp hsk = s\}|$ impractical even for small $n$. To remedy this we provide an algorithm to efficiently compute this statistic.
Secondly, to use the formula one has to compute the character table of the monoid. The Clifford-Munn-Ponizovskii Theorem (such as presented in \cite{ganyushkin2009irreducible}) gives an explicit description of the simple $\kbf M$-modules and technically makes the computation of the character table possible, provided that we know how to compute the simple modules associated to certain groups. However, this approach is rather convoluted and inefficient. We also note that although some results on character tables are known in the case of many interesting families of monoids, no algorithms are available to compute the character table of an arbitrary finite monoid. Thus, for the general case, we prove a formula for the character table that allows for computation exploiting the Green structure of the monoid for increased efficiency.

In the first section of this paper, we present the results necessary for our fixed-points counting algorithm, with a particular emphasis on the notions of Green classes and \schu groups.
In the second section, we prove a formula for the character table of a finite monoid, after recalling the necessary module and character theoretic results.
In the third section, we give and discuss the algorithms and equation systems used for computing the Cartan matrix with two focuses: the algorithms for fixed point counting and the equation system to compute the character table.
Finally, in the last section, we discuss the performance of these algorithms in terms of execution time and size of tractable problems.
	
	\section{Combinatorics of fixed point counting.}
	In this section we first recall essential and elementary results on the Green structure of finite monoids and on \schu groups. The informed reader may skip this first paragraph, with the exception of the notations (\ref{not:times}) that are used throughout this paper. We then use these results to devise a fixed point counting method.


In the totality of this paper, we assume that all monoids are finite. We will often use the following special case of finite monoid to illustrate the various results presented hereafter.

	\begin{dftn}[The full transformation monoid]
		Consider the set $T_n$ of all transformations of the set $\{1, \dots, n\}$, equipped with the multiplication given by map composition: $\forall f,g \in T_n, fg = f \circ g$.
		This is a monoid, aptly named the \emph{full transformation monoid}. A submonoid of $T_n$ is called a transformation monoid and $n$ is called its \emph{rank}.
	\end{dftn}
	
%

	\subsection{Green structure and \schu groups}\label{sec:green}
		Although finite monoids have been considered to be much wilder objects than groups, it turns out that, with the right optics, they are actually highly structured by their internal multiplication.
	Consider the divisibility relation: $x$ divides $y$ if $x = yz$ for some $z$. If $x, y, z$ are taken in a group $G$, the relation is trivial. If however, we take them in a general monoid $M$, left or right translation by an arbitrary element need not be surjective, making the question of $x \in M$ being a left (or right) multiple of $y \in M$ non-trivial. These questions of "divisibility" in a general monoid are studied under the name of Green structure, of which we give a brief overview necessary for our purpose in the subsection below. In the following subsection, we also present the related notion of \schu groups.
	
	
	\subsubsection{Green Structure}
	
	\begin{dftn}[Green's relations] \label{def:green_relations}
		Let $M$ be a finite monoid and $a, b$ two of its elements. The \emph{Green relations} are:
		\begin{itemize}
			\item The $\lc$ preorder is defined by $a \leql b \Leftrightarrow b=ua$ for some $u\in M$. The associated equivalence relation is : $a \in \lc(b) \Leftrightarrow a \leql b \textrm{ and } b \leql a \Leftrightarrow Ma = Mb$.
			\item The $\rc$ preorder is defined by $a \leqr b \Leftrightarrow b=av$ for some $v\in M$. The associated equivalence relation is : $a \in \rc(b) \Leftrightarrow a \leqr b \textrm{ and } b \leqr a \Leftrightarrow aM = bM$.
			\item The $\jc$ preorder is defined by $a \leqj b \Leftrightarrow b=uav$ for some $u, v\in M$. The associated equivalence relation is: $a \in \jc(b) \Leftrightarrow a \leqj b \textrm{ and } b \leqj a \Leftrightarrow MaM = MbM$.
			\item The $\hc$ equivalence relation is defined by $a \in \hc(b) \Leftrightarrow a \in \lc(b) \textrm{ and } a \in \rc(b)$.
		\end{itemize}
	\end{dftn}
	
	It can be proven (see for instance \cite[Theorem 1.9]{pin}) that in finite monoids, the relation $\jc$ is the smallest equivalence relation containing $\rc$ and $\lc$. This is not true in general, and this smallest relation is usually denoted by $\dc$ in the literature. Since we are only interested in finite monoids, we shall only use the terms $\jc$-relation, $\jc$-class, etc. From the definition, it is clear that the relation $\lc$ and $\rc$ are finer than $\jc$ and that $\hc$ is finer that both $\lc$ and $\rc$. Because of this, "the $\lc$-class of some $\hc$-class $H$" or "the $\jc$-class of some $\rc$-class $R$", etc... are well-defined and we shall denote them by $\lc(H), \jc(R),$ etc.
	
	\begin{lined}
		\begin{ex}[Green relations in $T_n$]\label{ex:green_tn}
			Let $a, b$ be two elements of $M$.
			\begin{itemize}
				\item If $a \lc b$, if and only if they have the same \emph{kernel} $\ker a = \{a\inv\{i\} \sepp i \in \intint{1,  n} \}$. We also say that $a$ and $b$ have the same nuclear equivalence.
				\item If $a \rc b$, if and only if they have the same image, $\im(a) = \im(b)$.
				\item Since $T_n$ is finite, $\jc$ is generated by $\lc$ and $\rc$ so $a\jc b$ if and only if $\im(a)$ and $\im(b)$ (or equivalently $\ker(a)$ and $\ker(b)$) have the same cardinality.
				\item Since $\hc$ is the intersection of $\lc$ and $\rc$, $a \hc b$ if and only $a$ and $b$ have the same image and the same kernel.
			\end{itemize}
			These conditions are necessary conditions in any transformation monoid. To get that they are sufficient, we use the fact that $\symm_n \subset T_n$ and that we can rearrange both image and kernel as we please.
			
			These relations are illustrated in the case of the monoid $T_3$ in Figure \ref{fig:green_t3}.
		\end{ex}
	\end{lined}	
			
	\begin{figure}[h!]
	\centering
	\begin{tikzpicture}[scale = 0.9]
	\tikzstyle{fleche}=[->,>=latex,rounded corners=4pt]
	\node at (2.2,0.25){$\jc$(1 2 3)};
	\node[font = \bfseries] at (0,0){1 2 3};
	\node at (0,-1/2){2 3 1};
	\node at (0,-1){3 1 2};
	\node at (1,0){2 1 3};
	\node at (1,-1/2){1 3 2};
	\node at (1,-1){3 2 1};

	\draw[black] (-.5,.3) to (1.5, .3);
	\draw[black] (-.5,.3) to (-.5, -1.3);
	\draw[black] (1.5,.3) to (1.5, -4.6);
	\draw[black] (-.5,-1.3) to (4.5, -1.3);
	
	\draw[fleche, red] (1.5,-1/2) -| (2, -1.3);
	\draw[fleche, red] (1.5,-1/2) -| (3, -1.3);
	\draw[fleche, red] (1.5,-1/2) -| (4, -1.3);
	\draw[fleche, green] (.5,-1.3) |- (1.5, -1.85);
	\draw[fleche, green] (.5,-1.3) |- (1.5, -2.95);
	\draw[fleche, green] (.5,-1.3) |- (1.5, -4.05);
	
	\node at (5.2,-1.6+0.25){$\jc$(1 2 2)};
	\draw[->,>=latex, bend left] (3,0.25) to (5.2,-1);
	\node[font = \bfseries] at (2,-1.6){1 2 2};
	\node at (2,-2.1){2 1 1};
	\node[font = \bfseries] at (3,-1.6){1 2 1};
	\node at (3,-2.1){2 1 2};
	\node at (4,-1.6){1 1 2};
	\node at (4,-2.1){2 2 1};
	
	\draw[black] (1.5, -2.4) to (4.5, -2.4);
	
	\node[font = \bfseries] at (2,-2.7){1 3 3};
	\node at (2,-3.2){3 1 1};
	\node at (3,-2.7){1 3 1};
	\node at (3,-3.2){3 1 3};
	\node[font = \bfseries] at (4,-2.7){1 1 3};
	\node at (4,-3.2){3 3 1};
	
	\draw[black] (1.5, -3.5) to (4.5, -3.5);
	
	\node at (2,-3.8){2 3 3};
	\node at (2,-4.3){3 2 2};
	\node[font = \bfseries] at (3,-3.8){3 2 3};
	\node at (3,-4.3){2 3 2};
	\node[font = \bfseries] at (4,-3.8){2 2 3};
	\node at (4,-4.3){3 3 2};
	
	\draw[black] (1.5, -4.6) to (5.5, -4.6);
	\draw[black] (2.5, -1.3) to (2.5, -4.6);
	\draw[black] (3.5, -1.3) to (3.5, -4.6);
	\draw[black] (4.5, -1.3) to (4.5, -6.4);
	
	\draw[rounded corners=4pt, red] (2, -4.6) |- (4.55, -6.7);
	\draw[rounded corners=4pt, red] (3, -4.6) |- (4.55, -6.7);
	\draw[rounded corners=4pt, red] (4, -4.6) |- (4.55, -6.7);
	\draw[fleche, red] (4.5, -6.7) -| (5, -6.4);
	\draw[rounded corners=4pt, green] (4.5, -1.85) -| (6.2, -4.6);
	\draw[fleche, green] (6.2, -4.6) |- (5.5, -4.9);
	\draw[fleche, green] (6.2, -4.6) |- (5.5, -5.5);
	\draw[rounded corners=4pt, green] (4.5, -2.95) -| (6, -4.6);
	\draw[fleche, green] (6, -4.6) |- (5.5, -4.9);
	\draw[fleche, green] (6, -4.6) |- (5.5, -6.1);
	\draw[rounded corners=4pt, green] (4.5, -4.05) -| (5.8, -4.6);
	\draw[fleche, green] (5.8, -4.6) |- (5.5, -6.1);
	\draw[fleche, green] (5.8, -4.6) |- (5.5, -5.5);
	
	\node at (6.2,-6.4){$\jc$(1 1 1)};
	\draw[->,>=latex, bend left] (5.7,-1.6) to (6.4,-6.1);
	\node[font = \bfseries] at (5, -4.9){1 1 1};
	\draw[black] (4.5, -5.2) to (5.5, -5.2);
	\node[font = \bfseries] at (5, -5.5){2 2 2};
	\draw[black] (4.5, -5.8) to (5.5, -5.8);
	\node[font = \bfseries] at (5, -6.1){3 3 3};
	
	\draw[black] (5.5, -4.6) to (5.5, -6.4);
	\draw[black] (4.5, -6.4) to (5.5, -6.4);
	
	\end{tikzpicture}
	\caption[Green relations in $T_3$.]{Green relations in $T_3$. \\ {\small Each block is a $\jc$-class, each line is a $\rc$-class, each column a $\lc$-class and each case an $\hc$-class. The red, green and black arrows represent the $\lc$, $\rc$ and $\jc$-order respectively.}}
\end{figure}
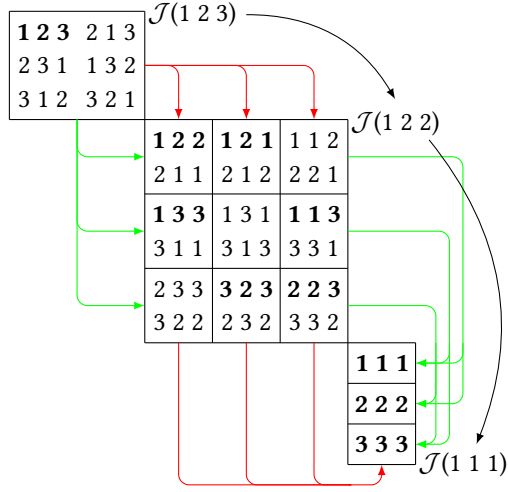\label{fig:green_t3}
			
	The following notations will prove useful, as the study of Green relations is, in part, the study of the maps given by left and right translations in the monoid.
	\begin{notation}\label{not:times}
		Let $h, k$ be elements of $M$ and $S$ be a subset of $M$. We denote by:
		\begin{itemize}
			\item $h\mul{S}$ the application from $S$ to $hS$ defined by $s \mapsto hs$,
			\item $\mul{S}k$ the application from $S$ to $Sk$ defined by $s \mapsto sk$,
			\item ${}_M\stab(S) = \{m \in M \sepp mS = S\}$,
			\item $\stab_M(S) = \{m \in M \sepp Sm = S\}$,
			\item $\fix_S(h, k) = \{s \in S \sepp hsk=s\}$.
		\end{itemize}
	\end{notation}
	
	Using these notations, let us recall Green's Lemma, which is one of the central elements of the theory of Green relations, as it shows that the structure of the relations is actually heavily constrained, making their study practical.
	
	\begin{lemme}[Green's Lemma]\label{lem:green}
		Let $a, a'$ be two elements in the same $\lc$-class and let $\lambda, \lambda'$ such that $\lambda a = a'$ and $\lambda' a' = a$. Then $\lambda\mul{\rc(a)}$ and $\lambda'\mul{\rc(a')}$ are reciprocal bijections. Moreover, for any $\lc$-class $L$ $\lambda\mul{\rc(a)\cap L}$ and $\lambda'\mul{\rc(a')\cap L}$ are reciprocal bijections.
		
		Similarly, if $a, b$ are two elements in the same $\rc$-class and $\rho, \rho'$ are such that $\rho a = b$ and $\rho' b = a$, then $\mul{\lc(a)}\rho$ and $\mul{\lc(b)}\rho'$ are reciprocal bijections. Moreover, for any $\rc$-class $R$ $\mul{\lc(a)\cap R}\rho$ and $\mul{\lc(b)\cap R}\rho'$ are reciprocal bijections.
	\end{lemme}
	
	\begin{figure}[h!]
	\centering
	\begin{tikzpicture}[scale = 0.9]
	\tikzstyle{fleche}=[->,>=latex,rounded corners=4pt]
	\node at (0,0){$a'$};
	\node at (0,3){$a$};
	\node at (5,0){$b'$};
	\node at (5,3){$b$};
	\draw (-.5,3.5) -- (-.5, -.5);
	\draw (.5,3.5) -- (.5, -.5);
	\draw (5.5,3.5) -- (5.5, -.5);
	\draw (4.5,3.5) -- (4.5, -.5);
	\draw (-.5,3.5) -- (5.5, 3.5);
	\draw (-.5,2.5) -- (5.5, 2.5);
	\draw (-.5,.5) -- (5.5, .5);
	\draw (-.5,-.5) -- (5.5, -.5);
	
	\draw[->, >=latex, bend right] (-.5,3) to (-.5,0);
	\node at (-1.7, 1.5){$\lambda\mul{\rc(a)}$};
	\draw[->, >=latex, bend right] (5.5,0) to (5.5,3);
	\node at (6.7, 1.5){$\lambda'\mul{\rc(a')}$};
	\draw[->, >=latex, bend left] (0,3.5) to (5,3.5);
	\node at (2.5, 4.6){$\mul{\lc(a)}\rho$};
	\draw[->, >=latex, bend left] (5,-.5) to (0,-.5);
	\node at (2.5, -1.6){$\mul{\lc(b)}\rho'$};
	
	\node at (-1.2, 3){$\rc(a)$};
	\node at (-1.2, 0){$\rc(a')$};
	\node at (0, 4){$\lc(a)$};
	\node at (5, 4){$\lc(b)$};
	
	\end{tikzpicture}
	\caption{Green's Lemma}
\end{figure}
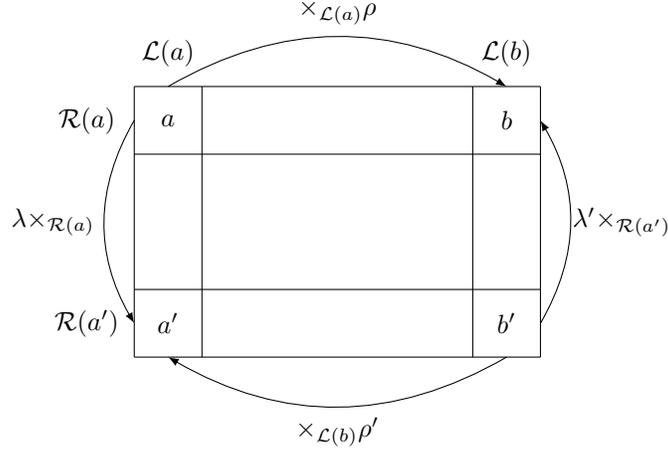
	
	An important consequence of Green's Lemma is that $\jc$-classes can be neatly organized as \emph{eggbox pictures}\footnote{Terminology introduced in \cite{clifford1961algebraic}}: the $\jc$-class can be represented as a rectangular array with the $\lc$-classes as columns, the $\rc$-classes as rows and the $\hc$-classes, the eggs, in the cases, as can be seen in Figure \ref{fig:green_t3}. This level organization is actually what allows for efficient computer representation of monoids and most of their algorithmic exploration. 
	
	\subsubsection{\schu groups}
	
	The Green structure offers a second way of facilitating computer exploration of monoids through groups that arise as stabilizers of some Green classes. These are called the \emph{\schu groups} and -- this a running theme of monoid theory -- allow for a number of monoid theoretic questions to be formulated in terms of groups for which we dispose of an array of efficient algorithms.
	
	\begin{dftn}[\schu groups]\label{def:schu_groups}
		Let $H$ be an $\hc$-class. The set $\{ h\mul{H} \sepp h \in {}_M\stab(H)\}$ equipped with map composition is a subgroup of $\symm(H)$ called the \emph{left \schu group} and denoted by $\Gamma(H)$.
		
		Similarly, $(\{ \mul{H}k \sepp k \in \stab_M(H)\}, \circ)$ is a subgroup of $\symm(H)$ called the \emph{right \schu group} and denoted by $\Gamma'(H)$.
	\end{dftn}

	\begin{lined}
		\begin{ex}\label{ex:schu_as_symm}
			Consider $H = \hc([1\ 3\ 1])$ (the elements of $T_n$ are given in function notation in all examples). We have :
			\[{}_M\stab(H) = \{[1\ 2\ 3], [3\ 2\ 1], [1\ 3\ 3], [3\ 1\ 1], [1\ 1\ 3], [3\ 3\ 1]\}\]
			and subsequently, $\Gamma(H) = \{[1\ 2\ 3]\mul{H}, [3\ 3\ 1]\mul{H}\}$. Notice that, as elements of $\Gamma(H)$, $[3\ 3\ 1]\mul{H} = [3\ 2\ 1]\mul{H}$ and that the only important thing is the permutations induced by the elements of ${}_M\stab(H)$ on $\im H$. Thus, in the case of transformation monoids, the left \schu group of an $\hc$-class $H$ can be represented as a subgroup of $\symm(\im H)$. In the same way, the right \schu groups can be represented as subgroups of $\symm(\ker H)$. This fact is used to represent the \schu groups in Section \ref{sec:Algos}.
		\end{ex}
	\end{lined}

	Our precedent remark on exploiting \schu groups to get efficient algorithms for computational monoid theoretic questions is seconded by the fact that \schu groups do not contain any "superfluous information" in the following sense.
	
	\begin{prop}\label{prop:schu_acts_freely}
		Let $H$ be an $\hc$-class. The natural actions of  $\Gamma(H)$ and $\Gamma'(H)$ on $H$ are free and transitive.
	\end{prop}
	
	We reproduce below a proof for Proposition \ref{prop:schu_acts_freely} from \cite{schu} for the purpose of showcasing the main argument. The argument itself is widely known and we will use it multiple times in the remainder of this paper.
	
	\begin{proof}
		 Two elements $h, h' \in H$ are in the same $\lc$-class so there is some $u \in M$ such that $uh=h'$. By Green's Lemma, this means that $u \in {}_M\stab(H)$, so $\Gamma(H)$ acts transitively on $H$. Suppose that $uh = h$ for some $u\in M$. Since $h, h'$ are also in the same $\rc$ class, there is some $v$ such that $h' = hv$, so $uh'=uhv=hv=h'$ : an element of $\Gamma(H)$ either fixes all points in $H$ or fixes none. The only element of $\Gamma(H)$ that fixes all points (and, consequently, the only one that fixes any point) is the identity and thus the action is free. The same arguments apply for $\Gamma'(H)$.
	\end{proof}
	
	A special case that is interesting to note, and that will be important later, it the case where $H$ is the $\hc$-class of an idempotent :
	
	\begin{dftn}\label{dftn:idem_and_maxgrp}
		An element $e \in M$ is \emph{idempotent} if $e^2=e$. Given an idempotent $e$, the set $G_e = \{x \in M\sepp \exists y \in M, xy=yx=e\}$ is called the \emph{maximal subgroup at $e$}. One can check that $G_e$ is indeed a group and that $G_e = \hc(e)$.
	\end{dftn}

	In that case, $\Gamma(H)$ and $\Gamma'(H)$ can be defined as before, and are canonically isomorphic to $G_e$, simply because $G_e \subset {}_M\stab(H)$ naturally induce a map making it a subgroup of $\Gamma(H)$ and that since $G_e$ acts freely and transitively on $H$, this map must be injective and surjective (and similarly for $\Gamma'(H)$).
	
	\begin{lined}
		\begin{ex}\label{ex:iso_idem}
			Consider, in Example \ref{ex:green_tn}, $e = [1\ 2\ 2]$ and $H = \hc(e)$. $e$ is an idempotent, and, indeed, $H = G_e$ is group : setting $t = [2\ 1\ 1]$, we have $e^2=e, t^2=e$ and $et=te=t$.
			As noted in Example \ref{ex:schu_as_symm} :
			\[\Gamma(H) = \symm(\{1, 2\}), \qquad \Gamma'(H) = \symm(\{\{1\}, \{2,3\}\}).\]
			Note that the canonical isomorphism between $\Gamma(H)$ and $\hc(e)$ is simply given by $g \in \Gamma(H) \mapsto g \cdot e \in H$.
		\end{ex}
	\end{lined}
	
	\subsection{Counting fixed points}\label{sec:fixed}
		Consider the problem of counting the number of elements of the set $\fix_{G}(h, k)$ 
	where $G$ is a finite group and $h, k \in G$. 
	If $\fix_{G}(h, k)$ is non-empty, it contains an element $\gamma$ such that $h\gamma k = \gamma$, or equivalently $h = \gamma k\inv \gamma\inv$. So for any $g \in \fix_G(h, k)$ we have:
	\[hgk = g \Leftrightarrow \gamma k\inv \gamma\inv g k = g \Leftrightarrow \gamma\inv g k = k\gamma\inv g.\]
	This means that $g \in \gamma C_G(k)$ where $C_G(k)$ is the centralizer of $k$ in $G$. Because the other inclusion is obvious, we get a description of $|\fix_{G}(h, k)|$: either $h$ and $k\inv$ are conjugated in which case there are $|C_G(k)|$ fixed points, or they are not, and there are no fixed points.
	In the case of a monoid, this reasoning mostly breaks: we crucially used the invertibility property, which monoids lack. The \schu groups seem to be ideal candidates to get back some of this invertibility. 
	In this section we clarify the role of the \schu groups for counting fixed points, how to give meaning to "$h\mul{H}$ and $\mul{H}k$ are in the same conjugacy class", and how to factorize our previous remark over all the $\hc$-classes of the same $\jc$-class.
	
	As the bijections between $\lc$ (and $\rc$) classes will play an major role in the remainder of this section, we introduce the following notations.
	\begin{notation}
		Given $R, R'$ two $\rc$-classes in the same $\jc$-class, we say that $(\lambda, \lambda')$ is a \emph{left Green pair} with respect to $(R, R')$ if:
		\begin{itemize}
			\item $\lambda R = R'$ and $\lambda'R' = R$.
			\item $(\lambda\lambda')\mul{R} = \Id_R$ and $(\lambda'\lambda)\mul{R'} = \Id_{R'}$
		\end{itemize}
		Similarly, given two $\lc$-classes $L, L'$ in the same $\jc$-classes, $(\rho, \rho')$ is a \emph{right Green pair} with respect to $(L, L')$ if:
		\begin{itemize}
			\item $L\rho = L'$ and $L'\rho' = L$.
			\item $\mul{L}(\rho\rho') = \Id_L$ and $\mul{L'}(\rho'\rho) = \Id_{L'}$
		\end{itemize}
	\end{notation}
	
	Using Green pairs, one can transport the problem of counting fixed points in an arbitrary $\hc$-class to a reference $\hc$-class.
	
	\begin{prop}\label{prop:transition}
		Let $H_1, H_2 \subset J$ be two $\hc$-classes contained in the same $\jc$-class. Let $\lambda, \lambda', \rho, \rho'$ such that:
		\begin{itemize}
			\item $(\lambda, \lambda')$ is a left Green pair with respect to $(\rc(H_1), \rc(H_2))$,
			\item $(\rho, \rho')$ is a right Green pair with respect to $(\lc(H_1), \lc(H_2))$.
		\end{itemize}
		Finally, let $(h, k) \in {}_M\stab(\rc(H_2)) \times \stab_M(\lc(H_2))$ and define $(h', k') = (\lambda' h \lambda, \rho k \rho')$. Then the maps $x \mapsto \lambda'x\rho'$ and $x \mapsto \lambda x\rho$ are reciprocal bijections between the sets $\fix_{H_2}(h, k) = \{a \in H_2 \sepp hak = a\}$ and $\fix_{H_1}(h', k')=\{a \in H_1 \sepp h'ak' = a\}$.
	\end{prop}
	
	\begin{proof}
		First notice that Green's Lemma give us the existence of $\lambda, \lambda', \rho, \rho'$ respecting the hypothesis we demand, and also gives that $x \mapsto \lambda'x\rho'$ and $x \mapsto \lambda x\rho$ are reciprocal bijections between $H_1$ and $H_2$.
		Let $a_1$ be an element of $H_1$ and denote by $a_2 = \lambda a_1 \rho$. Then:
		\[ha_2k = a_2 \Leftrightarrow \lambda'ha_2k\rho' = \lambda'a_2\rho' \Leftrightarrow (\lambda' h \lambda)a_1(\rho k \rho') = a_1 \Leftrightarrow h'a_1k' = a_1\]
		so these bijections restrict to $\fix_{H_1}(h', k')$ and $\fix_{H_2}(h, k)$.
	\end{proof}
	
	
	Keeping in mind our computational goals, transporting the problem of counting fixed points from $H_2$ to $H_1$ is helpful, as for the price of 4 monoid multiplications, we can use a lot of precomputations specific to a particular $\hc$-class, avoiding the repetition of multiple similar computations for each $\hc$-class. 
	
	The question is now to determine the fixed points in a single $\hc$-class, using our previous remark on conjugacy. Let us first clarify the idea of elements of the left and right \schu groups being in the same conjugacy class.
	
	\begin{prop}\label{prop:bij_cano_conj}
		Given and $\hc$-class $H$, $a \in H$ and $g \in \Gamma(H)$, we define $\tau_a(g)$ as the unique element of $\Gamma'(H)$ such that $g \cdot a = a \cdot \tau_a(g)$. Then $\tau_a : \Gamma(H) \longrightarrow \Gamma'(H)$ is an anti-isomorphism\footnote{Note that some authors equip the right \schu group with reversed composition, and thus obtain an isomorphism instead of anti-isomorphism.}. Moreover $\tau_a$ gives rise to a bijection between the conjugacy classes of $\Gamma(H)$ and $\Gamma'(H)$ that is independent of the choice of $a$.
	\end{prop}
	
	\begin{proof}
		The first part is known since \cite{schu}. We want to check that for $a \in H, g \in \Gamma(H)$, the conjugacy class of $\tau_a(g)$ is defined independently of $a$. Take any $a, b \in H$. By definition of $\Gamma(H)$, there exist some $h \in \Gamma(H)$ such that $b = h \cdot a$. So :
		\[b \cdot \tau_a(g) = (h \cdot a) \cdot \tau_a(g) = h \cdot (g \cdot a) = hgh\inv\cdot(h \cdot a) = b \cdot \tau_b(hgh\inv). \]
		Since $\Gamma'(H)$ acts freely, this means that $\tau_a(g) = \tau_b(h)\tau_b(g)\tau_b(h)\inv$ and thus $\tau_a(g)$ is conjugated with $\tau_b(g)$, which proves that the conjugacy class of $\tau_a(g)$ is indeed defined independently of $a$. Finally, as $\tau_a$ is a group morphism, the image are of two conjugated elements are conjugated, meaning that $\tau_a$ does indeed induces bijection between the conjugacy classes of the left and right \schu groups, independently of the choice of $a$.
	\end{proof}

	In the next proposition, we formalize the idea of searching the fixed points as some centralizer, but in the context of a monoid.

	\begin{prop}\label{prop:centralisateur}
		Let $H$ be a $\hc$-class, $a\in H$ and $(h, k) \in {}_M\stab(\rc(H)) \times \stab_M(\lc(H))$. Then
		\[|\fix_H(h, k)| = \left\{\begin{aligned}
		&|C_{\Gamma'(H)}(\mul{H}k)| & & \textrm{ if } \tau_a(h\mul{H})\inv \in \overline{\mul{H}k}\\
		&0 & & \textrm{ otherwise}
		\end{aligned}\right.\]
		where $\overline{\mul{H}k}$ is the conjugacy class of $\mul{H}k$ in $\Gamma'(H)$ and $C_{\Gamma'(H)}(\mul{H}k)$ is the centralizer in $\Gamma'(H)$ of $\mul{H}k$.
	\end{prop}
	
	\begin{proof}
		For simplicity, we commit an abuse of notation by denoting $h\mul{H}$ as $h$ and $\mul{H}k$ as $k$. Let $a$ be any element of $H$. 
		\begin{align*}
		\fix_H(h,k) & = \{b \in H \sepp hbk=b\}\\
		& = \{a \cdot g\sepp g \in \Gamma'(H) \textrm{ and } ha\cdot gk = a\cdot g \}\\
		& = \{a\cdot g\sepp g \in \Gamma'(H) \textrm{ and } a \cdot \tau_a(h)gk = a\cdot g \}\\
		& = \{a\cdot g\sepp g \in \Gamma'(H) \textrm{ and } \tau_a(h)gk = g\}.
		\end{align*}
		The last equality comes from the fact that $\Gamma'(H)$ acts freely so we can simplify the $a$. Suppose that $\fix_H(h, k)$ is non-empty and let $\gamma \in \Gamma'(H)$ such that $\tau_a(h)\gamma k = \gamma$. Then, for any $g \in \Gamma'(H)$ :
		\[\tau_a(h)gk = g \Leftrightarrow g \inv \tau_a(h)gk = e \Leftrightarrow g\inv\gamma k\inv\gamma\inv g k\ = e \Leftrightarrow [\gamma\inv g, k] = e\]
		where $[\cdot,\cdot]$ is the commutation bracket. This means that 
		\[\fix_H(h,k) = \{a \cdot g \sepp g \in \gamma C_{\Gamma'(H)}(k)\}.\]
		Note that because, again, $\Gamma'(H)$ acts freely, $\fix_H(h,k)$ has the same cardinality as $C_{\Gamma'(H)}(k)$ and that, from Proposition \ref{prop:bij_cano_conj} this is independent from the choice of $a$ which proves the result. 
	\end{proof}
	
	\begin{lined}
		\begin{ex}
			Consider $a = [1\ 2\ 2\ 3] \in T_4$ and $H = \hc(a)$. We have $\im a = \{1, 2, 3\}$ and $\ker a = \{\{1\},\{2,3\},\{4\}\}$. Notice that $H$ is not a group since $a^2 = [1\ 2\ 2\ 2] \notin H$. Considering the \schu groups as symmetric groups on the image and kernel common to all elements of $H$ as in Example \ref{ex:schu_as_symm}, we have $\Gamma(H) = \symm(\im a)$ and $\Gamma'(H) = \symm(\ker a)$.
			
			Let us first check for fixed points under the action of $h = [1\ 2\ 3\ 4]$ on the left and $k = [2\ 1\ 1\ 4]$ on the right. Seen as an element of $\Gamma(H)$, $h$ corresponds to $\Id_{\im a}$ and $k$ corresponds to $(\{2, 3\}\ \{1\})$ in $\Gamma'(H)$. Since we have $\tau_a(h) = \Id_{\ker a}$, it follows that $|\fix_H(h,k)| = 0$.
			
			If we now take $h$ to be $[1\ 3\ 2\ 4]$, the corresponding element in $\Gamma(H)$ is $(2\ 3)$ and $\tau_a(h) = (\{2, 3\}\ \{4\})$. Since $(\{2, 3\}\ \{4\})$ and $(\{2, 3\}\ \{1\})$ are conjugated in $\symm(\ker a)$, the set of fixed points is non-empty. Their centralizers have cardinal 2 and one can indeed check that $[2\ 3\ 3\ 1]$ and $[3\ 2\ 2\ 1]$ are the only fixed points in $H$.
		\end{ex}
	\end{lined}

	Putting together the previous results, we get the following Corollary on the cardinality of $\fix_J(h, k)$.
	
	\begin{cor}\label{cor:pt_fix_counting}
		Let $J$ be a $\jc$-class, $a$ any element of $J$ and denote by $H_0 = \hc(a), R_0 = \rc(a), L_0 = \lc(a)$. Let $h, k$ be any elements of $M$. We denote by:
		\begin{itemize}
			\item $(\lambda_R, \lambda'_R)$ a left Green pair with respect to $(R_0, R)$ for each $\rc$-class $R \subset J$,
			\item $(\rho_L, \rho'_L)$ a right Green pair with respect to $(L_0, L)$ for each $\lc$-class $L \subset J$,
			\item $S_{\rc}(h) = \{R \subset J \sepp R \textrm{ is a } \rc\textrm{-class and } hR = R\}$,
			\item $S_{\lc}(k) = \{L \subset J \sepp R \textrm{ is a } \lc\textrm{-class and } Lk = L\}$
		\end{itemize}
		Denoting the set of conjugacy classes of $\Gamma'(H_0)$ as $C$, we further define two vectors:
		\begin{itemize}
			\item $r_J(h) = (|C_{\Gamma'(H)}(g)|\cdot|\{R \in S_{\rc}(h)\sepp \tau_a(\lambda_R' h \lambda_R) \in \bar{g}\}|)_{\bar{g}\in C}$,
			\item $l_J(k) = (|\{L \in S_{\lc}(k)\sepp \rho_L' k \rho_L \in \bar{g}\}|)_{\bar{g}\in C}$. 
		\end{itemize}
		Then $\fix_J(h, k)$ has cardinality the dot product of $r_J(h)$ with $l_J(k)$.
	\end{cor}

	\section{Modules: character table and Cartan matrix}\label{sec:mod}
	In the remainder of this paper, $\kbf$ is a perfect field of null characteristic and $M$ is still a finite monoid. Furthermore, we suppose that $\kbf$ is "big enough", meaning that if not algebraically closed, at least a splitting field for the characteristic polynomials of the elements of $M$ seen as linear maps on $\kbf M$.

We will discuss the representation theory of a monoid $M$ over $\kbf$ using the language of modules. That is, a representation of $M$ over $\kbf$ will be a $\kbf$-vector space $V$ equipped with a linear action of the algebra $\kbf M$. As is usage, whenever the action is on the left we will say that $V$ is a $\kbf M-\module$ and a $\module - \kbf M$ if the action is on the right. If $M$, $M'$ are two finite, possibly different monoids, a $\kbf M - \module - \kbf M'$ is simply simultaneously a $\kbf M-\module$ and a $\module-\kbf M'$. For a monoid $M$, we denote by $M^{op}$ the \emph{opposite} monoid, with multiplication defined by $m\cdot_{M^{op}}m' = m'm$. We will use liberally the fact that a $\kbf M - \module - \kbf M'$ is naturally a $\kbf M \otimes \kbf M'^{op} - \module$ and a $\kbf(M \times M'^{op}) - \module$, and reciprocally. In the totality of this paper, we assume that the modules are finite dimensional as vector spaces over $\kbf$. Because of this, the Jordan-Hölder Theorem applies and the set of composition factors counted with multiplicities of a module is independent of the choice of a composition series. If $S$ is a $\kbf M$-module and $S$ is a simple $\kbf M$-module, we denote by $[V:S]$ the multiplicity of $S$ as a composition factor of $V$.

In this section, we deal with monoid representation theory, with the goal in mind to compute the character table of $M$. Using the Munn-Clifford-Ponizovskii, this can largely be reduced to group representation theory. Stated differently, the representation theory of a monoid $M$ is an extension of the representation theory of certain groups embedded in $M$. The groups in question are precisely the groups of Definition \ref{dftn:idem_and_maxgrp}.
In the first part of this section, we use this fact to find a description of an $\lc$-class containing an idempotent $e$ quotiented by its radical as a product of simple $\kbf G_e$-modules and simple $\kbf M$-modules. In the second part, we translate this decomposition in terms of characters, which gives us the formula we seek.
Finally, we recall and discuss the formula for the Cartan matrix from Thiéry \cite{Thiery.CartanMatrixMonoid}.

	\subsection{On modules}\label{sub:modules}
	Note that if we choose an element $a \in M$ and denote by $L$ its $\lc$-class and $H$ its $\hc$-class, we can equip $\kbf L$ with a $\kbf M-\module-\kbf\Gamma'(H)$ structure. $\kbf L$ is already a $\module-\kbf\Gamma'(H)$ by definition of $\Gamma'(H)$. We can also make it into a $\kbf M-\module$ by setting, for every $m \in M$ and $l \in L$:
\[m\cdot l = \left\{\begin{aligned}
& ml \textrm{ if } ml \in L\\
& 0 \textrm{ otherwise}
\end{aligned}\right..\]
This is well defined, as $ml \notin L$ implies that $l >_{\lc} ml$ and so for every $m' \in M$, $l >_{\lc} m'ml \notin L$ : once fallen out of $L$, we cannot climb back in.

We have previously stated that the representation theory of monoids is an extension of the representation theory of some subgroups. This mainly expressed using the two following functors.

\begin{dftn}
	Let $e \in M$ be an idempotent, $\lc(e)$ its $\lc$-class, $G_e$ the associated maximal subgroup. We define the two following maps :
	\begin{align*}
	\indu_{G_e}^M : & \left\{ \begin{aligned}
	\kbf G_e\mathrm{-mod} & \longrightarrow \kbf M\mathrm{-mod}\\
	V & \longmapsto \kbf\lc(e) \otimes_{\kbf G_e} V
	\end{aligned} \right.\\
	N_e : & \left\{ \begin{aligned}
	\kbf M\mathrm{-mod} & \longrightarrow \kbf M\mathrm{-mod}\\
	V & \longmapsto \{v \in V \sepp eMv = 0\}
	\end{aligned} \right..
	\end{align*}
\end{dftn}

The idempotents and their maximal subgroups play a central role in the theory. One can show (see for instance \cite[Proposition 1.14]{pin}) that if $e, f$ are two idempotents in the same $\jc$-class, there are some $x, x' \in M$ such that $xx' = e$ and $x'x = f$ and that $G_e \cong G_f$. A $\jc$-class containing an idempotent is called a \emph{regular} $\jc$-class.

We are almost ready to state the Clifford-Munn-Ponizovskii theorem, which is the central piece connecting group and monoid representation theory. We will need the notion of \emph{apex} of a $\kbf M$-module. A proof of the Clifford-Munn-Ponizovskii Theorem can be found in \cite[Section 5.2]{steinberg2016representation}.

\begin{dftn}
	Let $V$ be a $\kbf M-\module$, we denote its annihilator in $M$ by $\ann_M(V) = \{m \in M\sepp mV = \{0\}\}$. 
	This is clearly an two-sided ideal of $M$ and as such is an union of $\jc$-classes.
	A regular $\jc$-class $J$ is said to be the \emph{apex} of $V$ if $\ann_M(V) = I_J$ where $I_J = \{J \not \leql \jc(s) \sepp s \in M\}$. If $e \in J$ is an idempotent, we also say that $V$ has apex $e$.
\end{dftn}

\begin{thm}[Clifford-Munn-Ponizovskii]\label{thm:CMP}
	Let $M$ be a finite monoid, $e \in M$ an idempotent and $\kbf$ be a field.
	\begin{enumerate}
		\item There is a bijection between isomorphism classes of simple $M$-modules with apex $e$ and isomorphism classes of simple $G_e$-modules given by :
		\[V \longmapsto V^{\#} = \indu_{G_e}^M(V) / \rad(\indu_{G_e}^M(V)).\]
		The reciprocal bijection is given by $S \longmapsto eS$.
		\item $\rad(\indu_{G_e}^M(V)) = N_e(\rad(\indu_{G_e}^M(V)))$
		\item Every simple $M$-module has an apex.
		\item Every composition factor of $\indu_{G_e}^M(V)$ with the exception of $V^{\#}$ has an apex strictly $\jc$-greater than $e$. Moreover, $V^{\#}$ has apex $e$ and is a factor of multiplicity one.
	\end{enumerate}
\end{thm}

This allows us the following description of the $\lc$-class of an idempotent $e$.

\begin{prop}\label{prop:L_is_sum}
	Let $e \in M$ be an idempotent and $G_e$ be the maximal subgroup at $e$. Let $\irr_e$ be a set of representatives of the isomorphism classes of simple $\kbf G_e$-modules. Then: 
	$$\kbf \lc(e) = \bigoplus_{V\in \irr_e} \indu^M_{G_e}(V) \otimes_{\kbf} V^*$$
	where $V^*$ is the dual of $V$.
\end{prop}

\begin{proof}
	By definition, $\indu^M_{G_e}(V) = \lc(e) \otimes_{\kbf G_e} V$. Now, since direct sum and tensor product over a ring with identity commute :
	\[\bigoplus_{V\in \irr_e} \indu^M_{G_e}(V)\otimes V^* = \bigoplus_{V\in \irr_e} \kbf\lc(e) \otimes_{G_e} V\otimes V^* = \kbf\lc(e) \otimes_{G_e} \left(\bigoplus_{V\in \irr_e} V\otimes V^*\right)\]
	Because $\kbf$ is of null characteristic, $\kbf G_e$ is semi-simple. By the Wedderburn-Artin theorem, $\kbf G_e = \bigoplus_{V\in \irr_e} V\otimes V^*$ so :
	\[\bigoplus_{V\in \irr_e} \indu^M_{G_e}(V) \otimes V^* = \kbf\lc(e) \otimes_{\kbf G_e} \kbf G_e = \lc(e)\]
	since $\kbf G_e$ is a ring with identity.
\end{proof}

Note that this puts in relation three kinds of modules : the simple $\kbf G_e$-modules, which are well understood, $\kbf \lc(e)$ which is understood as well because it is a combinatorial module\footnote{That is, the multiplication of an element of the basis of the module by an element of $M$ is either an other element of the basis or 0.}, and finally the modules $\indu^M_{G_e}(V)$ which contain, in a sense, the simple $\kbf M$-modules that we are after. According to the Clifford-Munn-Ponizovskii Theorem, we still need to remove the radical of each $\indu^M_{G_e}(V)$ factor. Proposition \ref{prop:rad_is_ne} puts the radical in a form similar to Theorem \ref{thm:CMP} while Proposition \ref{prop:quotient_semisimple} does exactly this. Lemma \ref{lem:PG} and its Corollary are technical results on radicals used in the proof of Proposition \ref{prop:rad_is_ne}.

\begin{lemme}\label{lem:PG}
	Let $A, B$ be two finite dimensional algebras over a perfect field $\kbf$. Then: 
	\[\rad(A\otimes B) = \rad(A)\otimes B + A \otimes \rad(B).\]
\end{lemme}

While we haven't be able to find a source for that claim it seems to be folklore in the algebra representation community. For the sake of completeness, we reproduce a proof communicated to us by Pr. Pierre-Guy Plamondon\footnote{Pr. Pierre-Guy Plamondon, Laboratoire de Mathématiques de Versailles, Université Paris-Saclay, \href{https://www.imo.universite-paris-saclay.fr/~plamondon/}{website}.}.

\begin{proof}
	Because $\kbf$ is perfect, Wedderburn's Principal Theorem applies and we get the decompositions $A = A' \oplus \rad(A), B = B' \oplus \rad(B)$, with $A'$ and $B'$ semi-simple algebras. To prove the result, we show that $\rad(A)\otimes B + A \otimes \rad(B)$ is a nil radical and that \[\faktor{A \otimes B}{\rad(A)\otimes B + A \otimes \rad(B)}\] is semi-simple.
	Let us first show that the quotient is semi-simple.
	We have:
	\[
	\begin{aligned}
		A\otimes B & = (A' \oplus \rad(A)) \otimes (B' \oplus \rad(B))\\
		& = A'\otimes B' \oplus A'\otimes\rad(B) \oplus \rad(A)\otimes B' \oplus \rad(A)\otimes\rad(B).
	\end{aligned}
	\]
	On the other hand, the same decompositions give us 
	\[A \otimes \rad(B) \oplus \rad(A) \otimes B = A'\otimes\rad(B) \oplus \rad(A)\otimes B' \oplus \rad(A)\otimes\rad(B).\]
	Finally,
	\[\faktor{A\otimes B}{A \otimes \rad(B) \oplus \rad(A) \otimes B} = A' \otimes B'\]
	which, since $A', B'$ are semi-simple, is also semi-simple. $A \otimes \rad(B) \oplus \rad(A) \otimes B$ is also nil, because $\rad(B)$ and $\rad(A)$ are, so, indeed, $\rad(A \otimes B) = A \otimes \rad(B) \oplus \rad(A) \otimes B$.
\end{proof}

From Lemma \ref{lem:PG}, we get the following Corollary by recalling that if $V$ is a $A$-module, $\rad_A(V) = \rad(A)\cdot V$.

\begin{cor}
	Let $A, B$ be two finite dimensional unitary algebras over a perfect field. If $V_A \otimes V_B$ is a $A-\module-B$ (or equivalently a $A\otimes B^{op}-\module$), then 
	\[\rad_{A\otimes B^{op}}(V_A \otimes V_B) = \rad_A(V_A)\otimes B + A\otimes \rad_B(V_B).\]
\end{cor}

This allows us to identify the radical of $\kbf \lc(e)$.

\begin{prop}\label{prop:rad_is_ne}
	Let $M$ be a finite monoid, $e \in S$ an idempotent $G_e$ be the maximal subgroup at $e$ and $\kbf$ be a perfect field. 
	Then: 
	$$\rad_{\kbf M \otimes \kbf G_e^{op}}(\kbf \lc(e)) = N_e(\kbf \lc(e)).$$
\end{prop}

\begin{proof}
	Using Lemma \ref{lem:PG}, for $V$ a simple $G_e$-module, we have that :
	\[\begin{aligned}
	\rad_{\kbf M \otimes \kbf G_e^{op}}& (\indu^S_{G_e}(V) \otimes_{\kbf} V^*) \\
	= & \rad_{\kbf M} \indu_{G_e}^M(V) \otimes V^* + \indu_{G_e}^M(V) \otimes \rad_{\kbf G_e^{op}}(V^*)\\
	\overset{(1)}{=} & \rad_{\kbf M} \indu_{G_e}^M(V) \otimes V^*\\
	\overset{(2)}{=} & N_e(\indu_{G_e}^M(V)) \otimes V^*
	\end{aligned}\]
	where equality (1) comes from the simplicity of $V^*$ as a $\kbf G_e^{op}$-module and (2) is the second point of Theorem \ref{thm:CMP}.
	
	Since radical and direct sums commute, denoting by $\irr_e$ a set of representatives of the isomorphism classes of simple $\kbf G_e$-modules, we know that:
	\[\rad_{\kbf M \otimes \kbf G_e^{op}}(\lc(e)) = \bigoplus_{V \in \irr_e} N_e(\indu_{G_e}^M(V)) \otimes V^*.\]
	It remains to be seen why 
	\[\bigoplus_{V \in \irr_e} N_e(\indu_{G_e}^M(V)) \otimes V^* = N_e(\kbf\lc(e)).\]
	It is clear the direct sum on the left is a subset of the set on the right. 
	For the other inclusion, we see that if $V, V'$ are $\kbf M$-modules, $N_e(V \oplus V') = N_e(V) \oplus N_e(V')$. 
	Given the proposition \ref{prop:L_is_sum}, it is enough to show that $N_e(\indu_{G_e}^M(V)) \otimes V^* = N_e(\indu_{G_e}^M(V) \otimes V^*)$. 
	Let $x \in \indu_{G_e}^M(V) \otimes V^*$ be such that for every $m \in M, emx = 0$. $x$ can be writen as $\sum_i (\sum_j x_{i,j} b_j) \otimes b'_i$ where $\{b_j\}_j$ is a basis of $\indu_{G_e}^M(V)$ and $\{b'_i\}_i$ is a basis of $V^*$. For every $m \in M$, we have :
	\[em \cdot x = em \cdot \sum_i (\sum_j x_{i,j} b_j) \otimes b'_i = \sum_i (em \cdot \sum_j x_{i,j} b_j) \otimes b'_i = 0\]
	that is, for every $b'_i$ we get $em \cdot \sum_j x_{i,j}b_j = 0$ so $\sum_j x_{i,j}b_j \in N_e(\indu_{G_e}^M(V))$ which means $x \in N_e(\indu_{G_e}^M(V)) \otimes V^*$.
\end{proof}

\begin{prop}\label{prop:quotient_semisimple}
	Let $e \in M$ be an idempotent and $G_e$ be the maximal subgroup at $e$. Let $\irr_e$ be a set of representatives of the isomorphism classes of simple $\kbf G_e$-modules.
	Then:
	\[\kbf \lc(e) / \rad_{\kbf M \otimes \kbf G_e^{op}}(\kbf\lc(e)) \cong \bigoplus_{V \in \irr_e} V^{\#} \otimes V^*.\]
\end{prop}

\begin{proof}
	From Proposition \ref{prop:rad_is_ne}, we have a decomposition of $\rad_{\kbf M \otimes \kbf G_e^{op}}(\lc(e))$ as a direct sum adapted to the decomposition of $\kbf\lc(e)$ as $\bigoplus_{V\in \irr_e} \indu^S_{G_e}(V) \otimes_{\kbf} V^*$. So:
	\[\kbf \lc(e) / \rad_{\kbf M \otimes \kbf G_e^{op}}(\kbf\lc(e)) \cong \bigoplus_{V \in \irr_e} (\indu^S_{G_e}(V) \otimes_{\kbf} V^*) / (N_e(\indu^S_{G_e}(V))\otimes V^*).\]
	From Theorem \ref{thm:CMP}, we know that
	\[0 \longrightarrow N_e(\indu^S_{G_e}(V)) \longrightarrow \indu^S_{G_e}(V) \longrightarrow V^{\#} \longrightarrow 0\]
	is a short exact sequence. Since $N_e(\indu^S_{G_e}(V))\otimes V^*$ a submodule of $\indu^S_{G_e}(V)\otimes V^*$ and because tensor product is right exact we have a short exact sequence :
	\[0 \longrightarrow N_e(\indu^S_{G_e}(V))\otimes V^* \longrightarrow \indu^S_{G_e}(V)\otimes V^* \longrightarrow V^{\#}\otimes V^* \longrightarrow 0\]
	which proves the result.
\end{proof}
	
	\subsection{On characters}\label{sub:char}
	One of the major features of the finite group representation theory is the fact that all the information on a representation can be summarized in its \emph{character}. This (partially) carries over to monoid representation theory, as we shall see in this section where we reformulate the results of the previous section in terms of characters.

\begin{dftn}
	If $V$ is a finite dimension $\kbf M$-module, its \emph{character} is the map from $M$ to $\kbf$ defined by $\chi_{\kbf M}^V : m \longmapsto \tr(v \mapsto m \cdot v)$.
\end{dftn}

We recall the following well know facts about characters.  Proofs for fact 2 and 3 are respectively (ii) and (iii) of \cite[Proposition 7.12]{steinberg2016representation}\footnote{Note that for Fact 3, our reference deals only with the case $M = M'$, but the proof is the same.}.

\begin{prop}\label{prop:facts_on_char}
	\begin{enumerate}
		\item Let $V$ be a $\kbf M$-module. We have $\chi_{\kbf M}^V = \chi_{\kbf M^{op}}^{V^*}$.
		\item Consider the short exact sequence of $\kbf M$-modules :
		\[0 \longrightarrow A \longrightarrow B \longrightarrow B/A \longrightarrow 0.\]
		Then $\chi_{\kbf M}^{B/A} = \chi_{\kbf M}^B - \chi_{\kbf M}^A$.
		\item Consider $M, M'$ two finite monoids, $V$ a $\kbf M$-module and $W$ a $\kbf M'^{op}$-module. Then $\chi_{\kbf M \otimes \kbf M'}^{V \otimes W} = \chi_{\kbf M}^{V} \chi_{\kbf M'}^{W}$.
	\end{enumerate}
\end{prop}

The previous properties are simply extensions of similar properties on groups, and their proof is similar. From groups, we also keep in the case of monoids the linear independence of irreducible characters (see \cite[Theorem 7.7]{steinberg2016representation} for reference):

\begin{prop}\label{prop:char_libres}
	The irreducible characters $\{\chi_{\kbf M}^S \sepp S \textrm{ is a simple } \kbf M-\module\}$ are linearly independent as $\kbf$ valued functions.
\end{prop}

This, together with the second point in the Proposition \ref{prop:facts_on_char}, has a nice consequence. As we are interested in finite dimensional module over finite monoids, those modules have a composition series. Say that a $\kbf M$-module $V$, has $S$ as a composition factor with multiplicity $[V:S]$ for any simple $\kbf M$-module $S$. Then:
\[\chi_{\kbf M}^V = \sum_S [V:S]\chi_{\kbf M}^S.\]
In that way, since characters of the simple modules are linearly independent, the character of a module can be seen as a record of its composition factors.

The question of where to compute characters is worth asking: in the case of groups, one needs only to compute the character for a transversal of conjugacy classes to get its value everywhere. The case of monoids was described for the first time by McAlister in \cite{mcalister1972characters}. 

\begin{dftn}\label{def:conj}
	We say that two elements $m, m'$ in $M$ are in the same \emph{generalized conjugacy class} or \emph{character equivalency class} if for every $\kbf M$-module $V$, $\chi_M^V(m) = \chi_M^V(m')$. We note $C_M$ the set of generalized conjugacy classes.
\end{dftn}

\begin{prop}(\cite[Proposition 2.5]{mcalister1972characters})\label{prop:char_equiv}
	Let $\mathcal{E} = \{e_1, \dots, e_n\}$ be idempotent representatives of the regular $\jc$-classes of $M$ and for each $e_i$ let $\mathcal{C}_i = \{c_{i,1}, \dots, c_{i,m_i}\}$ be representatives of the conjugacy classes of $G_{e_i}$. Then the set $\mathcal{C}_M = \bigcup_{e_i \in \mathcal{E}} \mathcal{C}_i$ is a set of representatives of character equivalency classes of $M$.
\end{prop}

We can now recall the definition of the character table of a monoid.

\begin{dftn}
	Let $\irr_M$ be the set of isomorphism classes of simple $\kbf M$-modules and $C_M$ as in definition \ref{def:conj}.
	The \emph{character table} of $M$ over $\kbf$ is the (square) matrix defined by :
	\[X(M) = (\chi_{\kbf M}^V(m))_{V \in \irr_M, m \in C_M}.\]
	Moreover, if $e \in M$ is an idempotent, we define $X_e(M)$ as the matrix obtained by extracting from $X(M)$ only the lines corresponding to simple modules with apex $e$.
\end{dftn}

Finally, we can apply the language of characters to Proposition \ref{prop:quotient_semisimple}, which yield a formula for computing the character table of $M$ over $\kbf$ given the character tables of the groups $G_e$ over $k$.

\begin{prop}\label{prop:formula_char}
	Let $e \in M$ be an idempotent, $G_e$ be the maximal subgroup at $e$. 
	We have the formula for $X_e(M)$ : 
	\[X_e(M) = {}^tX(G_e)\inv \cdot \left(\chi_{\kbf M \otimes \kbf G_e^{op}}^{\kbf\lc(e)}(m, g) - \chi_{\kbf M \otimes \kbf G_e^{op}}^{N_e(\kbf\lc(e))}(m, g)\right)_{g \in C_{G_e}, m \in C_M}\]
	where the dot is the matrix product.
\end{prop}

\begin{proof}
	First, we have, because of Proposition \ref{prop:facts_on_char}-2, we have:
	\[\chi_{\kbf M \otimes \kbf G_e^{op}}^{\kbf\lc(e)/N_e(\kbf\lc(e))} = 
	\chi_{\kbf M \otimes \kbf G_e^{op}}^{\kbf\lc(e)}(m, g) - \chi_{\kbf M \otimes \kbf G_e^{op}}^{N_e(\kbf\lc(e))}(m, g)\]
	Then, from Proposition \ref{prop:quotient_semisimple}, we know that:
	\[\begin{aligned}
		\chi_{\kbf M \otimes \kbf G_e^{op}}^{\kbf\lc(e)/N_e(\kbf\lc(e))}(m, g) 
		& = \chi_{\kbf M \otimes \kbf G_e^{op}}^{\bigoplus_{V \in \irr_e} V^{\#} \otimes V^*}\\
		& = \sum_{V\in\irr_e} \chi_{\kbf M \otimes \kbf G_e^{op}}^{V^{\#} \otimes V^*}\\
		& = \sum_{V\in\irr_e} \chi_{\kbf M}^{V^{\#}}(m)\chi_{\kbf G_e}^V(g)\\
	\end{aligned}
	\]
	This last sum is clearly the dot product between the column of $X(G_e)$ indexed by $g$ and the column of $X_e(M)$ indexed by $m$. That is, the coefficient in position $(g, m)$ of $\chi_{M-G_e}^{\lc(e)/\rad(\lc(e))}$ is equal to the coefficient in position $(g, m)$ of ${}^tX(G_e)\cdot X_e(M)$, which, together with Proposition \ref{prop:facts_on_char}(ii), proves the equality.
\end{proof}
	
 	\subsection{One step beyond: The Cartan matrix}
	We are, at last, in measure to state the formula from Thiéry for the Cartan Matrix. Without getting into the specific details, the Cartan matrix can be seen as measure of how "not semi-simple" the algebra of the monoid is. 
We use a non standard definition of the Cartan matrix, first given in \cite[Definition 2.6]{Thiery.CartanMatrixMonoid}. A formal proof that this is equivalent to the usual definition can be found in \cite[Corolary 7.28]{steinberg2016representation}.

\begin{dftn}
	Let $\{S_1, \dots, S_n\}$ be a set of representatives of the isomorphism classes of simple $\kbf M$-modules. The simple $\kbf M \otimes \kbf M^{op}$ modules are the $S_i \otimes S_j^*$ for all $i, j \in \intint{1, n}$. Denote by $[\kbf M : S_i \otimes S_j^*]$ the multiplicity of $S_i \otimes S_j^*$ as a composition factor of $\kbf M$.
	
	The Cartan matrix of $\kbf M$ is defined by:
	\[C(\kbf M) = ([\kbf M : S_i \otimes S_j^*])_{i, j}\] 
\end{dftn}

In other words, the Cartan matrix is a recording of the multiplicities of the composition factors of $\kbf M$ as a $\kbf M \otimes \kbf M^{op}$ module. But so is its character! The difference being that the character of $\kbf M$ as it is computed is expressed in the basis of the character equivalency classes of $M \times M^{op}$ while the Cartan matrix is expressed directly in the basis of the simple modules. Since the basis change between the two is precisely given by the character table and hence, we have the Thiéry's Formula for the Cartan matrix.

\begin{prop}\label{prop:formule_cartan}
	The Cartan matrix is given by the formula:
	\[C(\kbf M) = {}^tX_M\inv B X_M\inv\]
	where $B = (|\{s \in M \sepp msm'\}|)_{m,m' \in C_M}$
\end{prop}
	
	\section{Some explicit computations}\label{sec:Algos}
	In this section, $M$ is a fixed submonoid of $T_n$. We explain how to combinatorially compute the characters of three interesting modules: 
$\kbf J$ for a $\jc$-class $J$, $\kbf M$ as $\kbf M-\module-\kbf M$ and $\kbf \lc(e)$ for an $\lc$-class containing an idempotent $e$ as a $\kbf M-\module-\kbf G_e$. For clarity, we will designate these characters as \emph{bicharacters}.

We begin by discussing the computational hypotheses we make before presenting the algorithms themselves.
In a third subsection, we detail how to compute the radical of $\kbf \lc(e)$ to apply our formula from Proposition \ref{prop:formula_char} to compute the character table of a monoid. This latter case is not as tidy as the former ones, as we don't have a purely combinatorial way of computing the radical.
	
	\subsection{Computational hypotheses}\label{sub:hypo}
	In this section we discuss the computational hypotheses necessary for the algorithms in the next section. 
This section is based on the work \cite{east2019computing} in which East, Egry-Nagy, Mitchell and Péresse provide efficient algorithms for all basic computational questions on finite semigroups (which include monoids). Although we limit our scope to the case of transformation monoids, methods described \cite{east2019computing} allow the algorithms described below to be applied to other interesting classes of monoids. Moreover, they can theoretically be applied to any finite monoid using a Cayley embedding in a full transformation monoid. In general however, this is very inefficient and not feasible in practice.

Following the authors of \cite{east2019computing}, we make the following fundamental assumptions that we can compute:  
\begin{itemize}
	\item Assumption I : a product of two elements of the monoid.
	\item Assumption II : the image and kernel of a transformation (note that we do not explicitly use this assumption, but that it is necessary for the algorithms of \cite{east2019computing} that we do use).
	\item Assumption III : Green pairs.
	\item Assumption IV : Given $h \in {}_M\stab(H)$ compute the corresponding element in $\Gamma(H)$ (understood as a permutation group of the image common to all elements of $H$ as seen in Example \ref{ex:schu_as_symm}), and similarly on the right.
\end{itemize}

Not only do we directly need to be able to do these computations for our own algorithm, but they are also prerequisite for the algorithms from \cite{east2019computing}. As such, we refer to the top of Section 5.2 of \cite{east2019computing} on how to realize these computations in the case of transformation monoids.

We, again, refer to \cite{east2019computing} for the specific algorithms meeting our computational prerequisites.
\begin{itemize}
	\item Computing the \schu groups: \cite[Algorithm 4]{east2019computing}
	\item Checking membership of an element in a Green class: \cite[Algorithms 7 \& 8]{east2019computing}.
	\item Finding idempotents: \cite[Algorithm 10]{east2019computing}. This algorithm also allows for finding the regular $\jc$-classes.
	\item Decomposing the monoid in $\rc, \lc$ and $\jc$-classes : \cite[Algorithm 11]{east2019computing} and its discussion. Note that by storing this decomposition, we can, given an element of the monoid, find the classes that contain it.
	\item Obtaining a representative of a Green class: this is given by the data structure representing the Green classes described at the top of \cite[Section 5.4]{east2019computing}.
\end{itemize}

Finally, we require the following points that, although they are not described in \cite{east2019computing}, are easily obtained from it.

\begin{itemize}
	\item Computing a set $C_M$ of character equivalency representatives: given Proposition \ref{prop:char_equiv}, this can be done in four steps:
	\begin{enumerate}
		\item compute a set $\mathcal{E}$ of idempotent representatives of the regular $\jc$-classes,
		\item compute $\Gamma(\hc(e))$ for each $e \in \mathcal{E}$,
		\item compute a set $C_e$ of representatives of the conjugacy classes of $\Gamma(\hc(e))$ for each $e \in \mathcal{E}$, using for instance the procedure described in \cite{hulpke2000conjugacy},
		\item for each $e \in \mathcal{E}$ and $c \in C_e$ compute the corresponding element of $\hc(e)$ as in Example \ref{ex:iso_idem}. 
	\end{enumerate}
	\item Computing $\tau_a$ as in Proposition \ref{prop:bij_cano_conj} :  given $g \in \Gamma(H)$, $\tau_a(g)$ is simply, seen as an element of $\symm(\ker a)$ :
	\[a\inv\{i\} \mapsto (g\cdot a)\inv\{g\cdot a(i)\},\]
	which can be computed in $O(n)$. Note that this is a special case of the application described in \cite[Proposition 3.11 (a)]{east2019computing}
	\item Testing that two elements $g, g'$ in $\Gamma'(\hc(a))$ are conjugated : $\Gamma'(H)$ is represented as a subgroup of $\symm(\ker a)$ and known procedures, such as the one described in \cite{butler1994inductive}, can be used.
	\item Computing the cardinality of a conjugacy class of a \schu group: for instance, the computer algebra system GAP uses the method described in \cite{hulpke2000conjugacy}.
\end{itemize}
	
	\subsection{Combinatorial bicharacter computing: 3 applications.}\label{sub:algos}
	We are now ready to present the algorithm for fixed-point counting, keeping in mind that we want first to use the formula from Section \ref{sec:mod}.\ref{sub:char} to compute the character table of the monoid and further to compute the Cartan matrix of the monoid. In the cases we are interested in, we use the formalism of character computing, since, as stated in the Lemma below, computing the characters of so called \emph{combinatorial modules} is actually counting fixed points.

\begin{lemme}\label{lem:fixed_is_char}
	Let $M, M'$ be two finite monoids and $(V, B)$ a finite dimensional $\kbf M-\module-\kbf M'$ space equipped with a basis $B$. If the actions of $M, M'$ on $(V, B)$ are \emph{combinatorial}, meaning for any $(m, b, m') \in M \times B \times M', mbm' \in (B \cup \{0\})$, then:
	\[\chi_{\kbf M \otimes \kbf M'^{op}}^V = |\{b \in B \sepp mbm' = b\}|.\]
\end{lemme}

\begin{proof}
	In the basis $B$, the matrix of the linear map $x \mapsto mxm'$ is a $\{0,1\}$-matrix, with for every $b \in B$ exactly one 1 in the $b$-th column, that 1 being on the $b$-th row if $mbm'=b$. Thus, the trace counts the number of fixed points.
\end{proof}

Note that we have already defined a structure of combinatorial $\kbf M-\module-\kbf G_e$ on $(\kbf\lc(e), \lc(e))$ for any idempotent $e$.
In the same way, $\kbf J$ for a $\jc$-class $J$, can be equipped with a structure of $\kbf M-\module-\kbf M$ by setting for every $(m, j) \in M\times J$:
\[m\cdot j = \left\{\begin{aligned}
						& mj \textrm{ if } mj \in J\\
						& 0 \textrm{ otherwise}
					\end{aligned}\right.
\textrm{ and }
j\cdot m = \left\{\begin{aligned}
						& jm \textrm{ if } jm \in J\\
						& 0 \textrm{ otherwise}
				  \end{aligned}\right..\]
As before, this is well defined: firstly because the actions on the left and on the right commute (because the monoid's law is associative by assumption) and secondly because either $m \leql m'$ or $m \leqr m'$ imply $m \leqj m'$ so, as in Section \ref{sec:mod}.\ref{sub:modules}, if $ml$ or $lm$ has "fallen to 0", it can't "climb back up" to $J$. 

This structure makes $(\kbf J, J)$ into a combinatorial module and we may apply our fixed points counting methods to compute its character.

\begin{algo}[Computing the bicharacter of a $\jc$-class]\label{algo:j_class}
	Keeping the assumptions and notations of the previous Paragraph \ref{sub:hypo}, we get from Corollary \ref{cor:pt_fix_counting} an algorithm to compute the bicharacter of $\kbf J$ as a $\kbf M-\module-\kbf M$:
	\begin{itemize}
		\item Input : A $\jc$-class $J$, a set of representatives of the character equivalency classes $C_M$.
		\item Output : A matrix $(|\{m \in J \sepp hmk=m\}|)_{(h, k) \in C_M^2}$
	\end{itemize}
	\begin{enumerate}
		\item Preparations:
		\begin{enumerate}
			\item Choose $a \in J$ and define $H = \hc(a)$. 
			\item Compute Green pairs $(\lambda_R, \lambda_R')$ (respectively $(\rho_L, \rho_L')$) for $(\rc(a), R)$ (resp.  $(\lc(a), L)$) for all $\rc$-class $R \subset J$ (resp. $\lc$-class $L \subset J$).
			\item Compute the set $C$ of conjugacy classes of $\Gamma'(H)$.
		\end{enumerate}
		\item For each character equivalency representative $h \in C_M$, initialize $r_J(h)$ and $l_J(h)$ to both be $(0)_{\bar{g}\in C}$.
		\begin{enumerate}
			\item For each $\rc$-class $R \subset J$, test if $h\lambda_Ra\in R$. If so, denoting by $\bar{g} $ the conjugacy class of $\tau_a((\lambda_R'h\lambda_R)\mul{H})$ in $\Gamma'(H)$, increment $r_J(h)$ by $|C_{\Gamma'(H)}(g)|$ at position $\bar{g}$.
			\item For each $\lc$-class $L \subset J$, test if $a\rho_Rh\in L$. If so, denoting by $\bar{g} $ the conjugacy class of $\mul{H}(\rho_L'h\rho_L)$ in $\Gamma'(H)$, increment $r_J(h)$ by $1$ at position $\bar{g}$.
		\end{enumerate}
		\item Compute the matrix $\chi = (r_J(h)\cdot l_J(k))_{(h, k) \in C_M^2}$ using the previously computed vectors and return $\chi$.
	\end{enumerate}
\end{algo}

\begin{lined}
	\begin{ex}\label{ex:aperiodic_bichar}
		An \emph{aperiodic} monoid is a monoid where all $\hc$-classes are singletons. Let us apply the algorithm we just described in the case of a $\jc$-class $J$ with trivial $\hc$-classes.
		Several simplifications occur : first, we don't need to check for the conjugacy class, as there is only one. Secondly, the conjugacy class has cardinality one.
		Consider the vectors $r_J = (|S_{\rc}(h)|)_{h \in C_M}$ and $r_J = (|S_{\lc}(h)|)_{h \in C_M}$ with $S_{\lc}(h)$ and $S_{\rc}(h)$ defined as in Corollary \ref{cor:pt_fix_counting}. The bicharacter is simply the matrix product of $r_j^T$ with $l_J$. The particular case of this algorithm for aperiodic monoid is described in \cite[Section .1]{Thiery.CartanMatrixMonoid}
	\end{ex}
\end{lined}

\begin{algo}[Computing the bicharacter of $\kbf M$]\label{algo:M_char}
	If we consider $(\kbf M, M)$ as a combinatorial $\kbf M-\module-\kbf M$, we immediately have that:
	\[\chi_{\kbf M \otimes \kbf M'^{op}}^{\kbf M} = \sum_{J \in \jc}\chi_{\kbf M \otimes \kbf M'^{op}}^{\kbf J}\]
	and we can therefore compute the bicharacter of the whole monoid $M$: we first compute a set $C_M$ of representatives of the character equivalency classes and we the iterate Algorithm \ref{algo:j_class} over all $\jc$-classes and sum the results. 

\end{algo}

The final useful example is the case of counting fixed points in a single regular $\lc$-class, for the purpose of computing the character table of the monoid.

\begin{algo}[Computing the bicharacter of an $\lc$-class]\label{algo:l_class}
	Let $e$ be an idempotent en let $L = \lc(e)$. In this example $\kbf L$ is still a combinatorial module, but it has the particularity, compared with the other two examples, that the monoids on the left and right are not the same. However, as the maximal subgroup at $e$, $G_e$, is a subsemigroup of $M$, the same results apply at no extra costs. 
	
	We can simply adapt the algorithm of Algorithm \ref{algo:j_class}. Since an element of $C_{G_e}$ acts "as itself" on the right, we don't need to keep track of the action of the right with a vector $r_L$ as we did previously.
	\begin{enumerate}
		\item Initialize $\chi$ to $(0)_{(h,k)\in C_M\times C}$
		\item For each $h \in C_M$, for each $\hc$-class $H$, test if $h\lambda_Ha\in H$. If so, denoting by $k$ the conjugacy class of $\lambda_H'h\lambda_H$ in $G_e$, increment $\chi$ by $|C_{G_e}(h)|$ at position $(h, k)$.
		\item Return $\chi$
	\end{enumerate}
\end{algo}

	\subsection{Computing $N_e(\kbf L)$}
	We are now almost in position to use the formula of Proposition \ref{prop:formula_char}: the character tables of the groups are supposed to be given, as we dispose of efficient group algorithms in the literature to compute them, from Algorithm \ref{algo:l_class} we now know how the efficiently compute the bicharacter of $\kbf \lc(e)$ as a $\kbf M \otimes \kbf G_e^{op}$-module for some idempotent $e \in M$. It remains to compute the bicharacter of $N_e(\kbf \lc(e))$ as a $\kbf M \otimes \kbf G_e^{op}$-module, which we discuss now.

Let $L = \lc(e)$. Recall that, by definition, $N_e(\kbf L) = \{x \in \kbf L \sepp  eMx = 0\}$.
Taking $L$ as a basis for $\kbf L$, we can form a matrix with rows indexed by $M \times L$ and columns indexed by $L$, with the coefficient at $((m, l), l') = 1$ if $eml = l'$ and 0 otherwise. Computing the kernel of this matrix yields a basis for $N_e(\kbf L)$ but is extremely inefficient as the number of rows is many times the cardinality of the monoïd.

Notice first that for any $m \in M$, $em \leqr e$ so we can consider only the elements of $M$ that are $\rc$-smaller than $e$. 
Conversely, recall that the structure of $\kbf M$-module on $M$ is defined by $m\cdot l = ml$ if $ml \in L$ and $0$ otherwise and that this latter case happens if $ml  \leql l$. Since if $m \leql l$ implies $ml \leql l$ we have that the $(m, l)$-th row of the matrix is null and that we may omit it. This shows that we need only to consider the element of $M$ that are not $\lc$-below $e$. Together with the previous point, this means that the similarly defined matrix but whose rows are only indexed by $\rc(e) \times L$ has the same kernel. This is good news, as we may now exploit the structure of the $\jc$-class given by Green's Lemma to further reduce the dimension of this matrix. Indeed, another consequence of Green's Lemma is the so called "Location Theorem" from Clifford and Miller. A proof can be found in \cite[Theorem 1.11]{Pin:Automata}

\begin{lemme}[Location Theorem]
	Let $r, l$ be two elements in the same $\jc$-class.
	We have:
	\[rl = \left\{\begin{aligned}
	& \gamma \in \rc(r) \cap \lc(l) \textrm{ if } \lc(r)\cap\rc(l) \textrm{ contains an idempotent},\\
	& \gamma <_{\jc} l, r \textrm{ otherwise}.
	\end{aligned}\right.\]
\end{lemme}

\begin{figure}[h!]\label{fig:location}
	\centering
	\begin{tikzpicture}[scale = 0.9]
	\tikzstyle{fleche}=[->,>=latex,rounded corners=4pt]
	\node at (0,0){$l$};
	\node at (0,3){$rl$};
	\node at (5,0){$e$};
	\node at (5,3){$r$};
	\draw (-.5,3.5) -- (-.5, -.5);
	\draw (.5,3.5) -- (.5, -.5);
	\draw (5.5,3.5) -- (5.5, -.5);
	\draw (4.5,3.5) -- (4.5, -.5);
	\draw (-.5,3.5) -- (5.5, 3.5);
	\draw (-.5,2.5) -- (5.5, 2.5);
	\draw (-.5,.5) -- (5.5, .5);
	\draw (-.5,-.5) -- (5.5, -.5);
	
	\draw[->, >=latex, bend left = 45] (-.3, 0) to (-.3, 3);
	\node at (-1.3, 1.5){$=$};
	\draw[->, >=latex, bend left] (5,2.7) to (0.3,0);
	\node at (3, 1){$\times$};
	
	\node at (-1.2, 3){$\rc(r)$};
	\node at (-1.2, 0){$\rc(l)$};
	\node at (0, 4){$\lc(l)$};
	\node at (5, 4){$\lc(r)$};
	\node at (6, 0){};
	
	\end{tikzpicture}
	\caption{Location Theorem\\ {\small Since there is an idempotent $e$ in $\lc(r)\cap \rc(l)$, $rl$ stays in the same $\jc$-class, in $\lc(l) \cap \rc(r)$.}}
\end{figure}

\begin{lemme}
	Let $e \in M$ be an idempotent, $R = \rc(e)$ its $\rc$-class and $R'$ another $\rc$-class of $\jc(e)$. Let $(\lambda, \lambda')$ be a left Green pair for $(L, L')$. Then $(\lambda e, e\lambda')$ is a left Green pair for $(R, R')$.
	
	Similarly, if $L = \lc(e)$, $L'$ is a $\lc$-class of $\jc(e)$ and $(\rho, \rho')$ is a right Green pair for $(L, L')$, then $(e\rho, \rho'e)$ is a right Green pair for $(L, L')$
\end{lemme}

\begin{proof}
	Let $g$ be any element of $\hc(e)$ and $g' = \lambda g$. Since $e$ is idempotent, $\hc(e)$ is a group with identity $e$ so we have $e\lambda'\lambda e g = e \lambda'\lambda g = eg = g$ and $\lambda ee\lambda' g' = \lambda eg = \lambda g = g'$ which, from Green's Lemma, make $(\lambda e, e\lambda')$ a left Green pair for $(L, L')$. A similar argument applies for the second part of the proposition.
\end{proof}

\begin{rmk}
	This lemma means that for a regular $\jc$-class $J$ and for any two $\lc$-class (or $\rc$-class) it contains, we may choose a corresponding Green pair among the elements of those two classes.
\end{rmk}

\begin{prop}\label{prop:equations}
	Let $e \in M$ be an idempotent and $H = \hc(e), L = \lc(e), R = \rc(e)$ and $J = \jc(e)$. 
	For each $\rc$-class $R' \subset J$, we choose a left Green pair $(l, l') \in J^2$. We denote by $\mathfrak{L}$ the set of all $l$ for the chosen left Green pairs. We define $\mathfrak{R}$ similarly.
	Then $N_e$ is the set of solutions of :
	\[\forall r \in \mathfrak{R}, \forall g \in H,\quad \sum_{l \in \mathfrak{L}} \ind_H(rl)x_{l(rl)\inv g} = 0\]
\end{prop}

\begin{proof}
	Consider an element $a \in R$. $a$ can be written in a unique way as $gr$, with $g \in G$ and $r \in \mathfrak{R}$ corresponding to $\lc(a)$. Similarly, an element $b$ in $L$ as a unique decomposition as $l\gamma, l \in \mathfrak{L}, \gamma \in H$.
	For an element $x \in \kbf L$ we note:
	\[x = \sum_{l \in \mathfrak{L}, \gamma \in H}x_{l\gamma}l\gamma\]
	its decomposition over the basis $L$.
	
	We want to find the equations that describe $\ker(gr\mul{L})$ (where $gr\mul{L}$ is the linear map on $\kbf L$ obtained by extending the monoid's multiplication by linearity). From the Location Theorem, we get that $\im(gr\mul{L}) \subset \kbf H$.
	For $k \in H$, denote by $f_{k, gr}$ the $k$-th coordinate function of $gr\mul{L}$.
	Because $gr\mul{L}$ acts combinatorially on $\kbf L$, we have :
	\[f_{k, gr}(x) = \sum_{l \in \mathfrak{L}, \gamma \in H} \ind_{\{k\}}(grl\gamma)x_{l \gamma}\]
	
	Note that $x_{l \gamma}$ appears in the sum if and only if $grl\gamma = k$. From the Location Theorem, and because we choose $l \in L, r \in R$, we have $grl\gamma = k \textrm{ if and only if } rl \in H \textrm{ and } \gamma = (rl)\inv g\inv k$
	and thus the equation becomes:
	\[f_{k, gr}(x) = \sum_{l \in \mathfrak{L}} \ind_H(rl)x_{l(rl)\inv g\inv k}.\]
	For $x$ to be in $\ker(gr\mul{L})$, $x$ must cancel simultaneously $f_{k, gr}$ for all $k\in H$. We now have a set of equations for $ker(gr\mul{L})$, and we can deduce that the set of equations
	\[\forall r \in \mathfrak{R}, \forall g,k \in H, \quad f_{k, gr}(x) = \sum_{l \in \mathfrak{L}, \gamma \in H} \ind_H(rl)x_{l (rl)\inv g\inv k} = 0\]
	describes $N_e(\kbf L)$.
	However, the equation system is redundant as the equation $f_{k, gr}(x) = 0$ is the same for all pairs $(g, gk')$ with $k' \in H$. Removing the duplicates equation gives the system announced in the proposition.
\end{proof}

\begin{lined}
	\begin{ex}[$N_e$ in the case of an aperiodic monoid]
		As in Example \ref{ex:aperiodic_bichar}, let us consider the case of a $\jc$-class with trivial $\hc$-classes. In this case, we have $L = \mathfrak{L}, R = \mathfrak{R}$ and $H = \{1_H\}$, so the equations become:
		\[\forall r \in R, \quad \sum_{l\in M}\ind_H(rl)x_l.\]
		Again from the Location Theorem, we have that $\ind_H(rl) = 1$ if and only if there is an idempotent in $\lc(r)\cap\rc(l)$. So if we form a matrix $A$ with rows indexed by $L$ and columns indexed by $R$, and with coefficients 1 at $(\lc(r),\rc(l))$ if $\lc(r)\cap \rc(l)$ contains an idempotent and 0 otherwise, the above equations becomes :
		\[(x_l)_{l\in L}^T A = 0,\]
		that is, in the case of an $\hc$-trivial $\jc$-class, $N_e(\kbf L)$ is the left kernel of the eggbox picture seen as a $\{0, 1\}$-matrix. 
	\end{ex}
\end{lined}

Note that given this set of equations, we can compute the character $\chi_{\kbf M \otimes \kbf G_e^{op}}^{N_e(\kbf\lc(e))}$ from the formula in Proposition \ref{prop:formula_char} using classical linear algebra algorithms to find a basis of $N_e(\kbf\lc(e))$ and then computing the value of the character at any $(m, g) \in C_M \times C_{G_e}$ by iterating over the basis vectors, applying $(m, g)$ as a linear map and computing the relevant coefficient in the image vector.

	\section{Performances, computational complexity and benchmarks}\label{sec:perf}
	In this section, we discuss performances in terms of complexity whenever we can, and in terms of benchmarks for timings and memory usage. In the next paragraph, we discuss the challenges in measuring performances and the subsequent choices made. Given these considerations, in the three paragraphs following, we discuss the computationnal performances of our three main objects of interest: the combinatorial bicharacter (\emph{i.e.} fixed-point counting), the character table and finally the Cartan matrix.
	
	\subsection{Discussion and Challenges}\label{sub:challenges}
	


Despite concentrating on transformation monoids, monoids being as diverse as they are, a meaningful analysis of the complexities of the above algorithms is difficult to do in terms of elementary parameters such as the rank (\emph{i.e.} the cardinality of the set acted upon), the number of generators or even the cardinality. Using only these parameters, we would have to make so many simplifying assumptions as to lose all meaning in our analysis.

Indeed, for our algorithms, the relevant metrics tends to be related to the Green-class structure and are the number of $\lc$ and $\rc$-classes in a $\jc$-class, the cardinality of the $\hc$-classes of a $\jc$-class as well as the number of conjugacy classes in its \schu groups, the number of generalized conjugacy classes, etc. In some sense, these metrics vary a lot: two transformation monoids acting on the same number of points with the same number of generators can have vastly different Green structures. 
Given this hurdle, it seems that the most meaningful level of analysis in terms of complexity is at the level of the $\jc$-class, where these parameters are a constant, rather than on the level of the whole monoid where we do not have any good monoid-wide estimate of how these useful metrics average out on all (regular) $\jc$-classes. Thus, in the following subsections, we will comment in any depth on the time complexity only for Algorithm \ref{algo:j_class} and for the solving of the equation system from Proposition \ref{prop:equations}.

Although a theoretical complexity of our algorithms is hard to precisely provide on the monoid-wide level, the real test of viability is to see if the algorithm effectively terminates in practice on typical examples. Thus, we provide benchmarks (for time and memory usage) for the computation of the three main objects of our discussion: the bicharacter, the character table and the Cartan Matrix.
We would like to know what is the typical performance on a "randomly chosen finite monoid". 
However, the question of taking a "generic" transformation monoid acting on a set number of points is a subject in itself, and outside the scope of this paper. 
Thus we chose to use two families of test cases: firstly the full transformation monoids and secondly random monoids generated by $m$ elements picked uniformly at random in $T_n$ and denoted by $R(m, n)$.
The full transformation monoids constitute a interesting test case in the sense that they are as big as possible. However, they are also highly structured and many computer algebra systems, including GAP, are smart enough to detect that the \schu groups are actually symmetric groups and thus use some non general algorithms that could not be used on typical finite monoids. This may introduce bias in our measurements on the full transformations monoids (and indeed probably does given Figure \ref{fig:plot_bichar}), hence why we separate the measure made on them from those made on the $R(m,n)$ monoids.


The measurements provided for these new algorithms (as well as the computation results presented hereafter) all come from an implementation using the computer algebra system GAP.
All performance measures are realized on a laptop equipped with an Intel Core i7-10850H @ 2.7GHz (on one core) and 16 GB memory. The measures on random monoids presented in this section are realized on monoids of the form $R(m,n)$, with $(m, n) \in \{(4,3), (5,3), (5,4), (6,5), (7,6), (9,8)\}$ (with $(9,8)$ excluded in the character table and Cartan matrices benchmarks) with 10 randomly chosen test cases for each $(m, n)$ in this set. Every appearance of $R(m, n)$ for a specific pair $(m, n)$ refers to the same monoid.
Our specific implementation, as well as the test cases used and the raw data, are publicly available on our git repository\footnote{\url{github.com/ZoltanCoccyx/monoid-character-table}}.

	\subsection{Fixed point counting}\label{sub:perf_count}
	In the case of Algorithm \ref{algo:j_class}, we can give some analysis of the time complexity in terms of the Green structure of the particular $\jc$-class Algorithm \ref{algo:j_class} is applied to.

\begin{prop}
	Consider a $\jc$-class $J$ containing $n_L$ $\lc$-classes, $n_R$ $\rc$-classes, containing an $\hc$-class $H$ with $n_C$ conjugacy classes in $\Gamma'(H)$, and let $C_M$ be a set representatives of the character equivalence classes, as before. Then, the Algorithm \ref{algo:j_class} does:
	\begin{itemize}
		\item $n_C$ cardinality computations of conjugacy classes of $\Gamma'(H)$ (assuming memoization to be able to do a lookup in step 2-a of Algorithm \ref{algo:j_class}, instead of computing it on the fly),
		\item $O(|C_M|(n_L + n_R))$ monoid multiplications, Green class membership tests and conjugacy class of $\Gamma'(H)$ membership tests,
		\item $O(|C_M|n_R)$ computations of $\tau_a$,
		\item $O(|C_M|n_L)$ conjugacy class of $\Gamma'(H)$ cardinality lookups,
		\item $n_C|C_M|^2$ integer multiplications.
	\end{itemize}
\end{prop}

\begin{proof}
	This simply results from an inspection of Algorithm \ref{algo:j_class}.
\end{proof}

As explained before, we cannot meaningfully extend this analysis to Algorithm \ref{algo:j_class}. We can get a similar result by inspection of Algorithm \ref{algo:l_class}: it is essentially the same algorithm, except that it is applied on only one $\lc$-class and that we don't need the final integer multiplications ate the end.

\begin{figure}[h!]
	\centering
	
	\makebox[\textwidth][c]{
		\includegraphics[width = 0.6\textwidth]{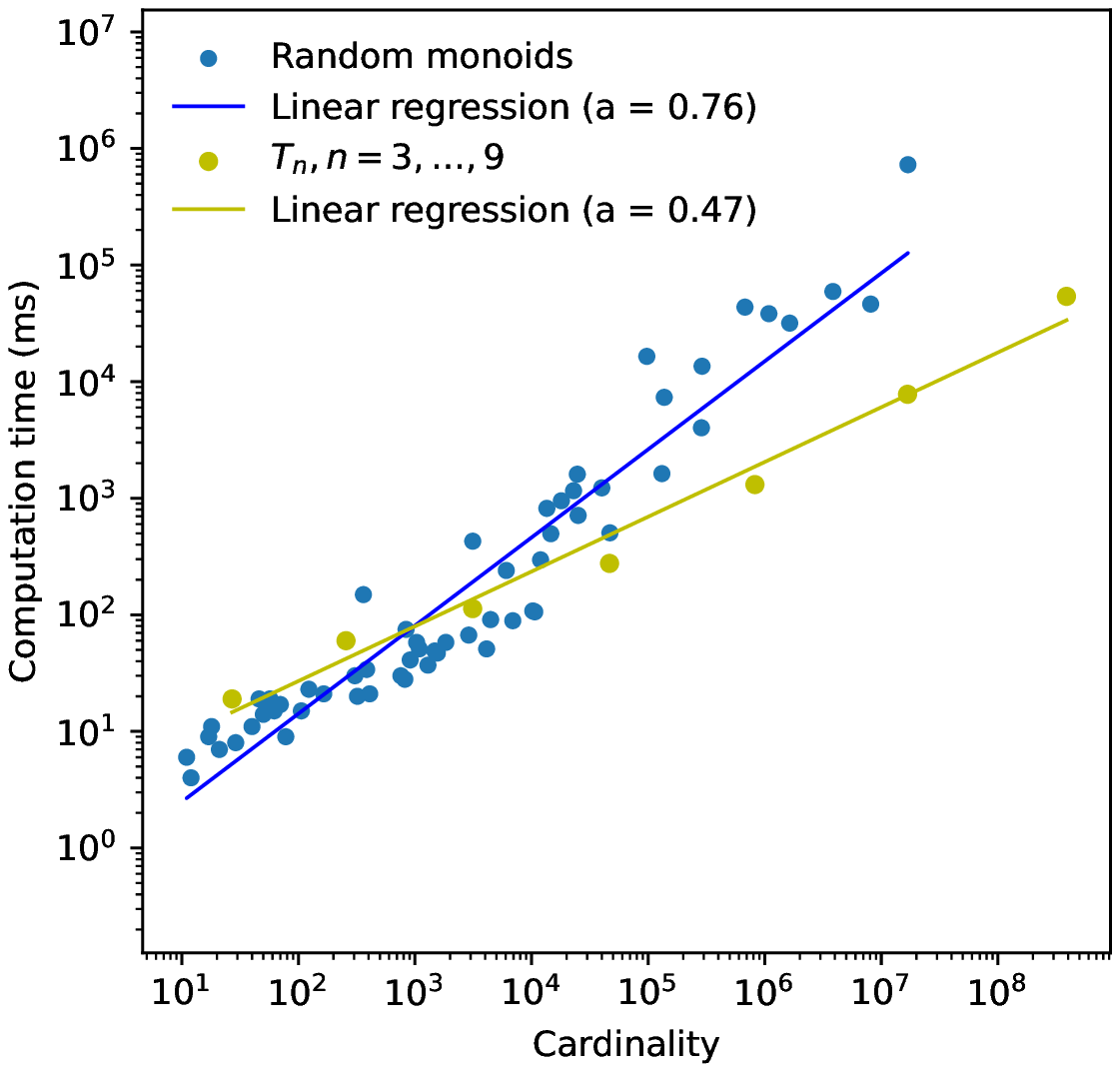}
		\hfill
		\includegraphics[width = 0.6\textwidth]{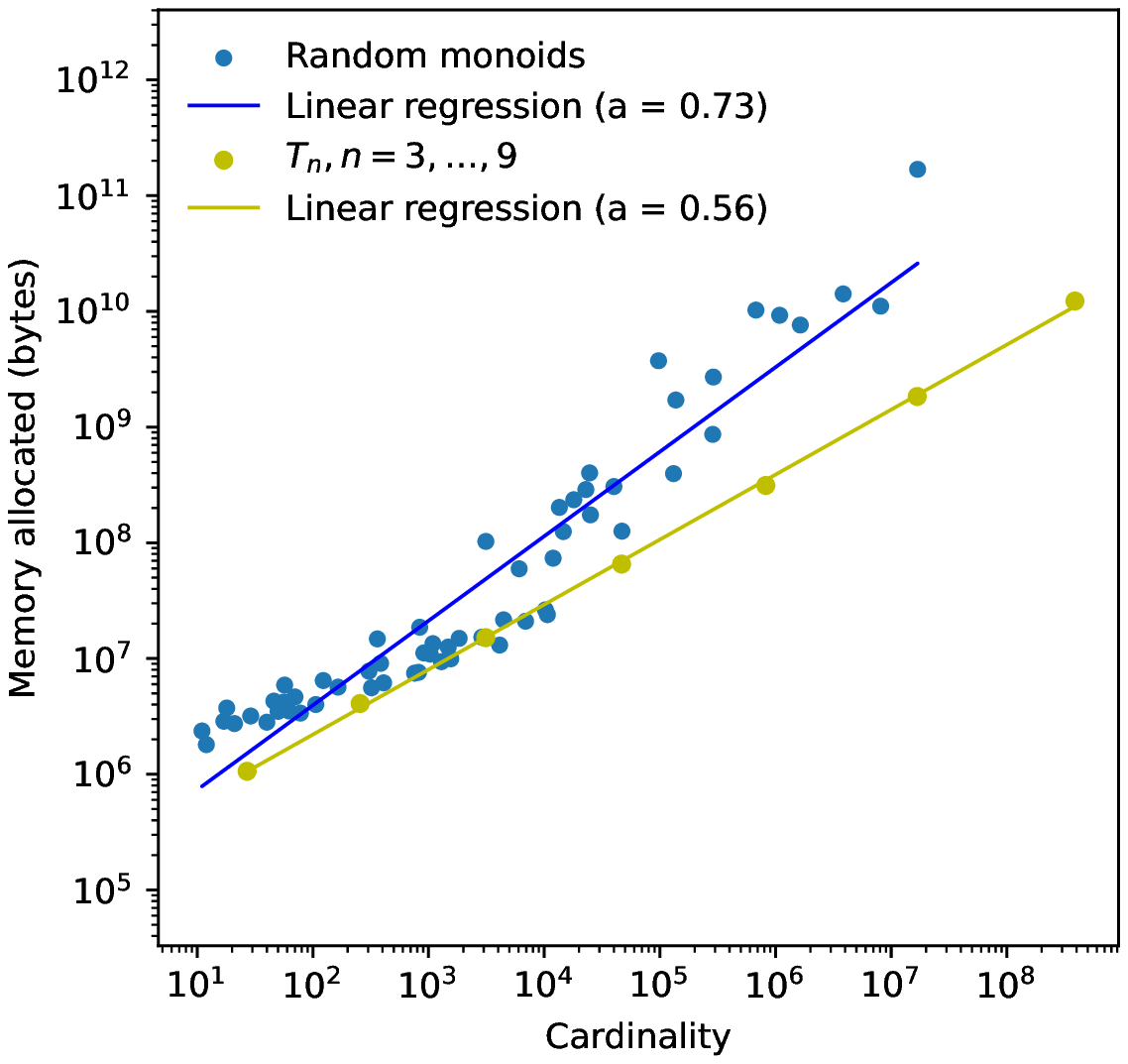}
	}
	
	\caption{Computation time and memory usage of Algorithm \ref{algo:M_char}.
		\\
		The blue points correspond to the random monoids, the yellow ones to $T_n$ for $n \in \intint{3, 9}$. The yellow points are excluded of the linear regression as the algorithm is "anormaly" efficient on them: we mesure a complexity on the full transformation monoids of approximately $O(\sqrt{|T_n|})$ while the measured complexity on random monoids is about $O(n^{0.76})$ in time and memory.}
	\label{fig:plot_bichar}
\end{figure}

Note that we do not provide a cumulative formula for the complexity of Algorithm \ref{algo:j_class} as for instance the complexity of a conjugacy class membership test heavily depends on the algorithm used by the computer algebra system, that can itself vary depending on the characteristics of the \schu groups. This makes the task of providing a meaningful evaluation of the global complexity of the algorithm quite difficult, mainly because expressing the complexity of those "elementary" operations of monoid multiplications, membership testing, etc... in terms of the same parameters is not straightforward and in some cases even unknown as noted in \cite{butler1994inductive}.
However, we can at least compare this to the naive algorithm of testing if every element of $J$ is a fixed point which demands $O(n_Ln_R|H|^2|C_M|^2)$ monoid multiplications: as long as the complexity of the more complex operations of Green class or conjugacy class membership testing remains limited in terms of monoid multiplications, our complexity is better. For instance, in the case of the monoid $T_n$, all the required operations can be done on $O(n)$, making Algorithm \ref{algo:j_class} (and, in turn, Algorithm \ref{algo:M_char}) more efficient than the naive algorithm, as can be seen in Table \ref{tab:bichar}, with a sub linear (with respect to cardinality) measured complexity (Figure \ref{fig:plot_bichar}).

\begin{table}[!h]
	\centering
	
	\begin{tabular}{c|c|c|c|c}
		Monoid & Cardinality & Coefficients & Naive & Ours \\
		\hline
		$T_3$ & 27 & $6^2$ & 29 ms & 18 ms\\
		$T_4$ & 256 & $11^2$ & 92 ms & 63 ms\\
		$T_5$ & 3125 & $18^2$ & 1.44 s & 113 ms\\
		
		$T_6$ & 46656 & $29^2$ & 53.0 s & 0.34 s\\
		$T_7$ & 823543 & $44^2$ & >30 min & 1.59 s\\ 
		$T_8$ & 16777216 & $66^2$ & $\cdots$ & 8.86 s\\ 
		$T_9$ & 387420489 & $96^2$ & $\cdots$ & 56.7 s\\
	\end{tabular}
	
	\caption{Computation time of the regular representation bicharacter.}
	\label{tab:bichar}
\end{table}
	
	\subsection{Character table}\label{sub:perf_table}

As shown in Table \ref{tab:char_tab}, the computation of the character table takes much longer. This is due to the fact that, to compute the radical of $\kbf\lc(e)$ for an idempotent $e$, we must solve a linear system of size $|\rc(e)|\times|\lc(e)|$ which necessitates $O(|\rc(e)|^2|\lc(e)|)$ arithmetic operations. In the case of the full transformation semigroup $T_n$, if $e$ as $k$ elements in its image, $|\lc(e)| = k! \times \binom{n}{k}$, while $|\rc(e)| = k! \times S(n,k)$ where $S(n, k)$ is a Stirling number of the second kind, which gives $|\rc(e)| \sim k^n$. The size of that linear system becomes rapidly untractable. Moreover, once we have a basis of $N_e(\kbf L)$ of cardinality $d$, we still have to compute the $C_M^2$ character values in $O(d^2)$ operations each. Experiments show that the computation time of the character tables of the maximal subgroups is small in comparison of all radical related computations. As can be seen on Figure \ref{fig:plot_char_table}, time and memory usage are in lockstep (at least for big enough monoids) and the limiting factor is memory (the test on random monoids fail for the random monoids of the form $R(9,8)$ by exceeding the 16GB memory capacity of our testing machine).

\begin{table}[!h]
	\centering
	
	\makebox[\textwidth][c]{
		\begin{tabular}{c|c|c|c}
			Monoid & Cardinality & Coefficients & Ours\\
			\hline
			$T_3$ & 27 & $6^2$ & 27 ms \\
			$T_4$ & 256 & $11^2$ & 151 ms\\
			$T_5$ & 3125 & $18^2$ & 1.74 s\\
			$T_6$ & 46656 & $29^2$ & 29.8 s\\
			$T_7$ & 823543 & $44^2$ & 11.0 min\\  
		\end{tabular}
	}
	
	\caption{Computation time of the character table.}
	\label{tab:char_tab}
\end{table}

\vspace{-10mm}
\begin{figure}[h!]
	\centering
	
	\makebox[\textwidth][c]{
		\includegraphics[width = 0.6\textwidth]{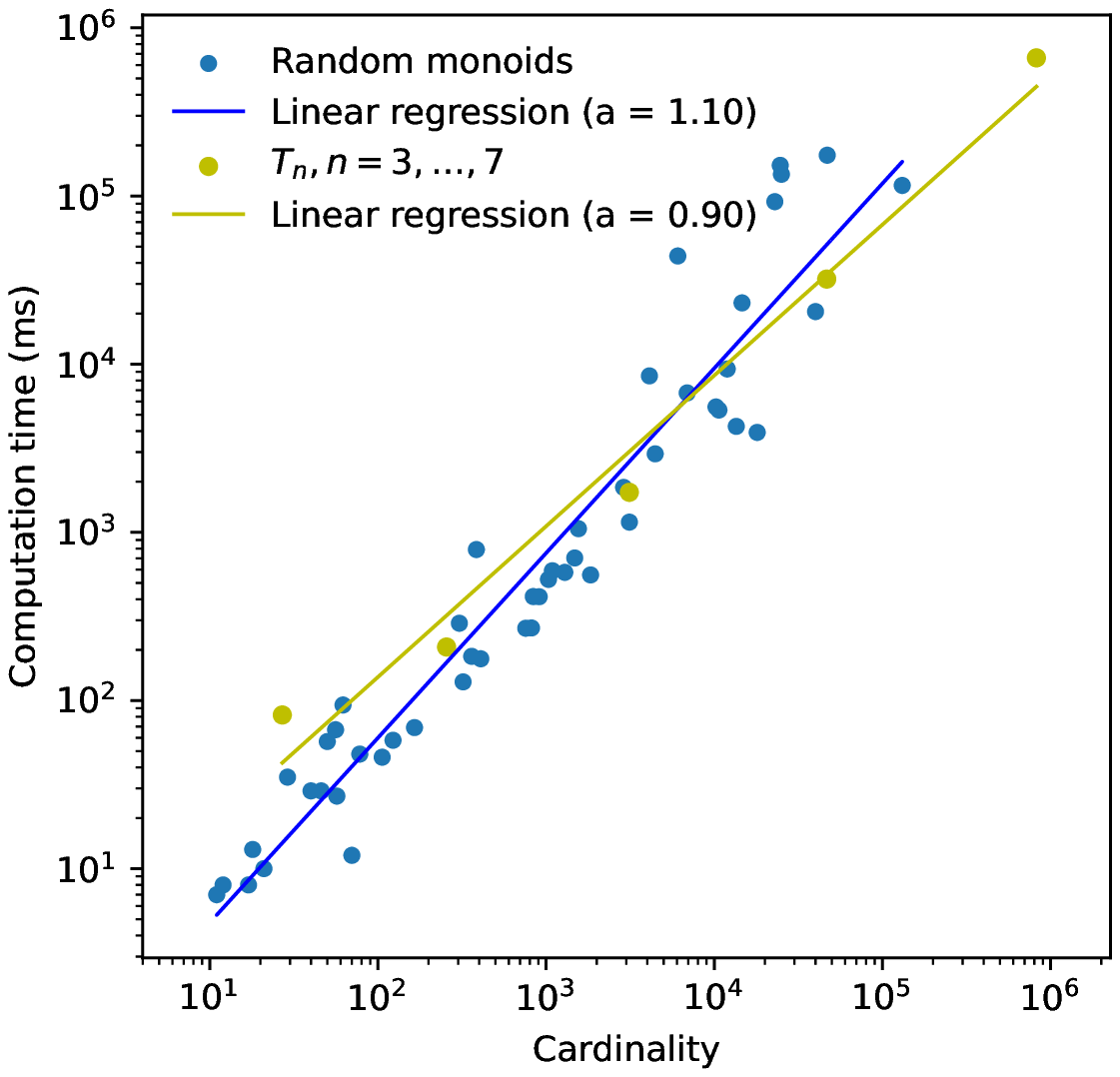}
		\hfill
		\includegraphics[width = 0.6\textwidth]{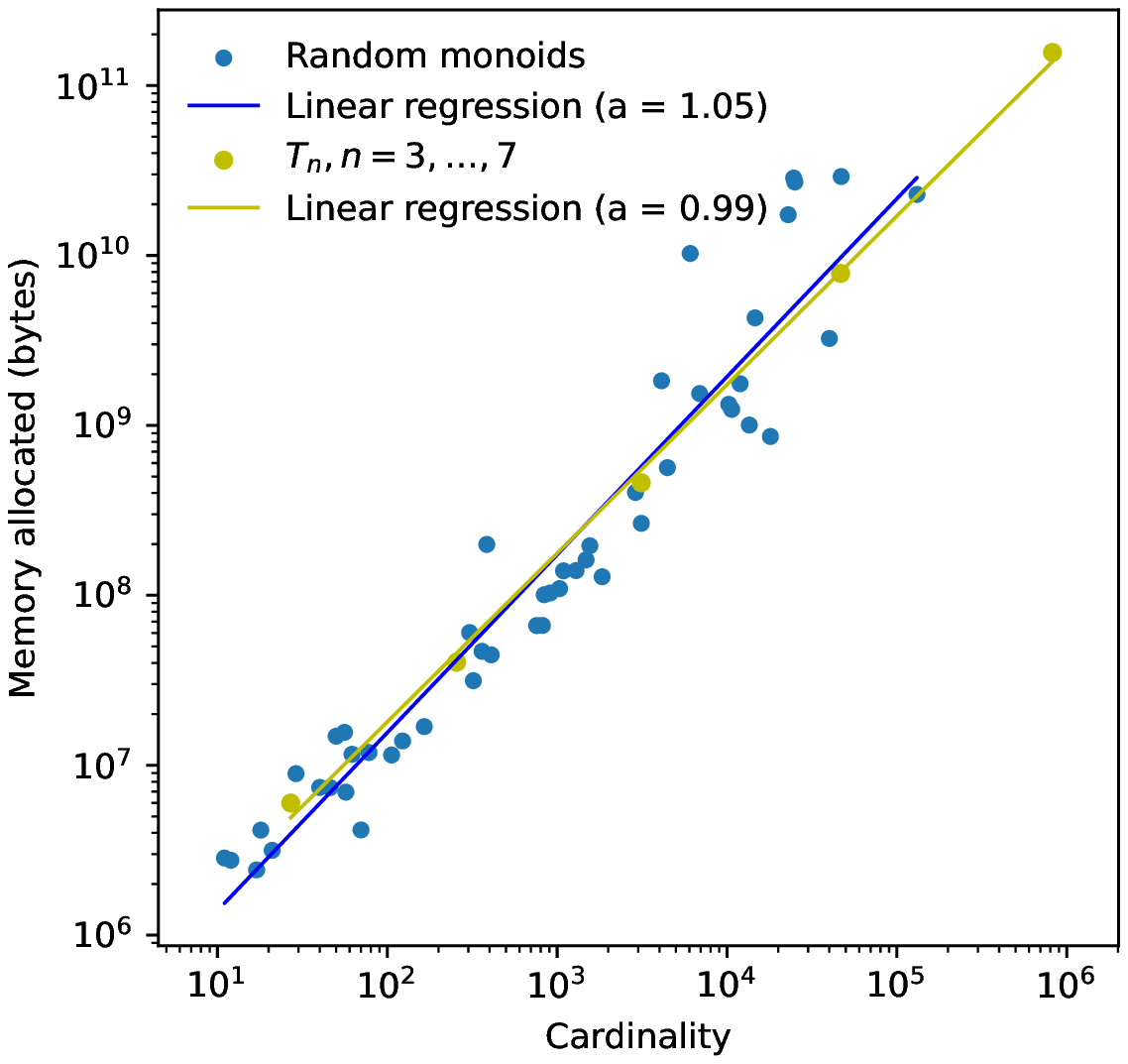}
	}
	
	\caption{Computation time and memory usage for computing the character table using Propositions \ref{prop:formula_char} and \ref{prop:equations}.
		\\
		The blue points correspond to the random monoids, the yellow ones to $T_n$ for $n \in \intint{3, 7}$. As before, the yellow points are excluded of the linear regression although in this case, $T_n$ behave more or less like the randomly chosen monoids. The measured complexity on random monoids is slighly more than linear in time and memory.}
	\label{fig:plot_char_table}
\end{figure}

%
	
	\subsection{Cartan matrix}\label{sub:perf_cartan}
		Finally, for the computation of the Cartan Matrix, the previous timings show that the vast majority of the computation time is spent computing the character table of the monoid. As the computation of the combinatorial bicharacter is more than a hundred times faster than the computation of the character table, this is a clear invitation to improve in particular the computation of the character of the radical of the $\lc$-classes. In Table \ref{tab:cartan}, we show some timings for that computation, and a comparison with Sage generalist algorithm (based on the Peirce decomposition of the monoid algebra) for the computation of the Cartan Matrix: despite its limitations our specialized algorithm allows for the handling of larger objects. Indeed, our algorithm has near linear performance with respect to cardinality, while sage's has roughly cubic complexity.

\begin{table}[!h]
	\centering
	
	\makebox[\textwidth][c]{
		\begin{tabular}{c|c|c|c}
			Monoid & Coefficients & Sage's & Ours\\
			\hline
			$T_3$ & $6^2$ & 575 ms & 56 ms \\ 
			$T_4$ & $11^2$ & 5.23 min & 173 ms \\
			$T_5$ & $18^2$ & {\color{red} >2h} & 1.82 s\\
			$T_6$ & $29^2$ & $\cdots$ & 34.6 s\\   
			$T_7$ & $44^2$ & $\cdots$ & 11.5 min\\
		\end{tabular}
	}
	
	\caption{Computation time of the Cartan matrix. \\ {\small In the case $T_5$, Sage's algorithm was interrupted before the end of the computation.}}
	\label{tab:cartan}
\end{table}

\begin{figure}[h!]
	\centering
	
	\makebox[\textwidth][c]{
		\includegraphics[width = 0.6\textwidth]{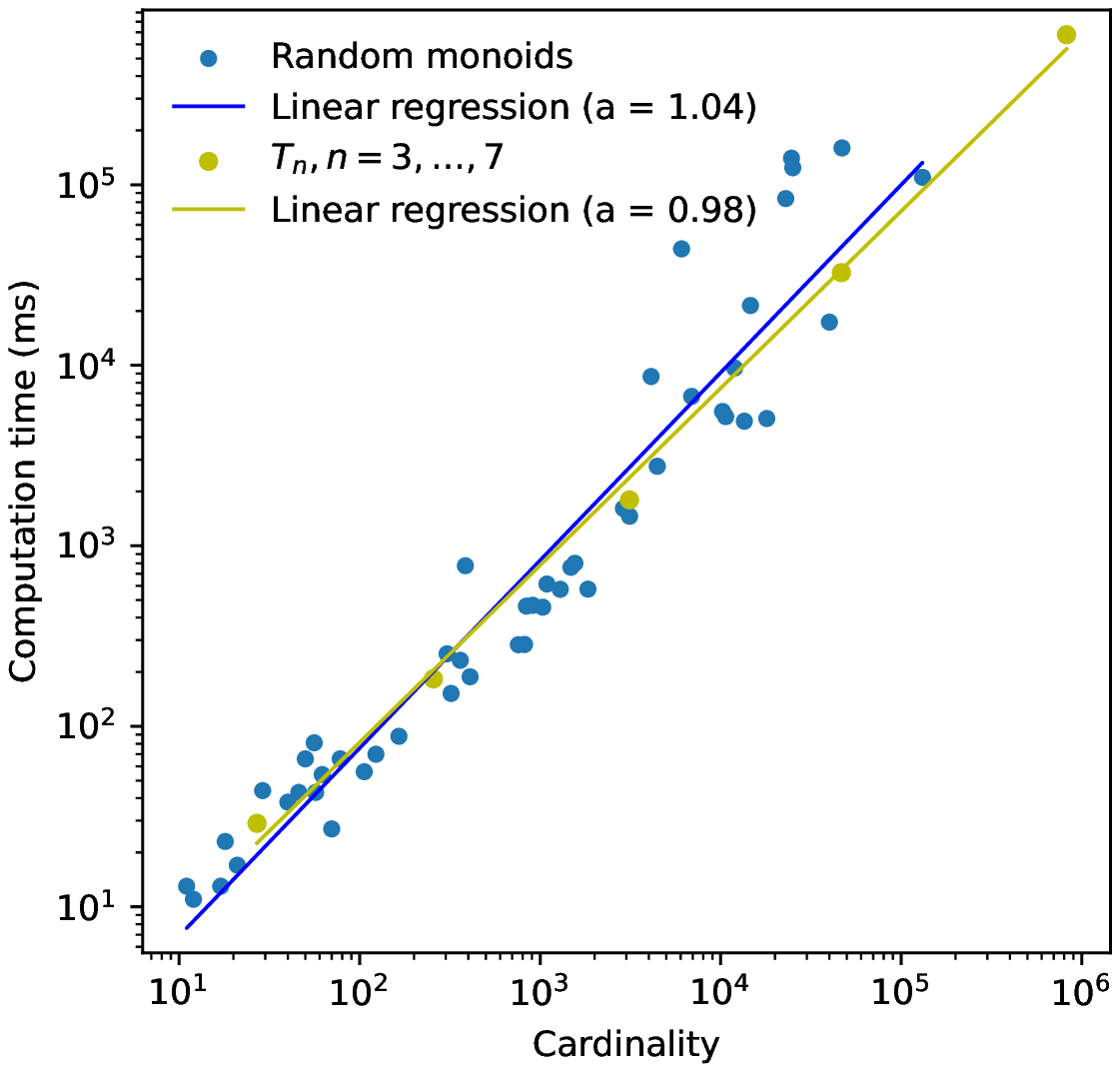}
		\hfill
		\includegraphics[width = 0.6\textwidth]{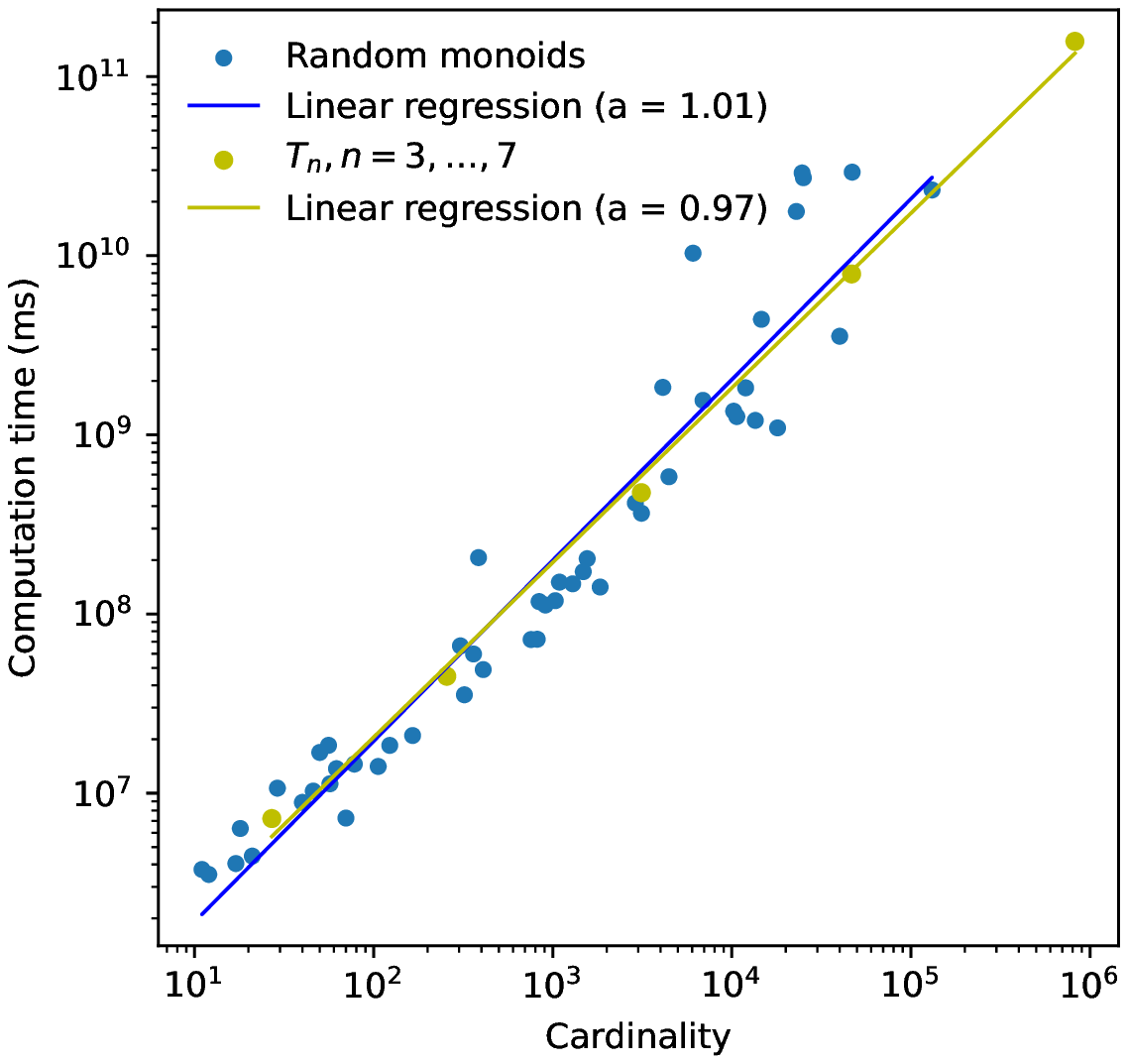}
	}
	
	\caption{Computation time and memory usage for computing the Cartan Matrix using Proposition \ref{prop:formule_cartan}.
		\\
		The blue points correspond to the random monoids, the yellow ones to $T_n$ for $n \in \intint{3, 7}$. As before, the yellow points are excluded of the linear regression although in this case, $T_n$ behave more or less like the randomly chosen monoids. The measured complexity on random monoids is slighly more than linear in time and memory.}
	\label{fig:plot_cartan}
\end{figure}

Again, time and memory usage are in lockstep, and although memory fails before time for $T_8$ and onward and for the monoids of the form $R(9, 8)$, using the regression we can predict a computation time of around 6 hours on our testing machine if it was not memory limited.

An example of a Cartan matrix obtained using our Algorithms and Thiéry's formula is pictured in Figure \ref{fig:cartanT7}. 

\begin{figure}[h!]
	\centering
	\includegraphics{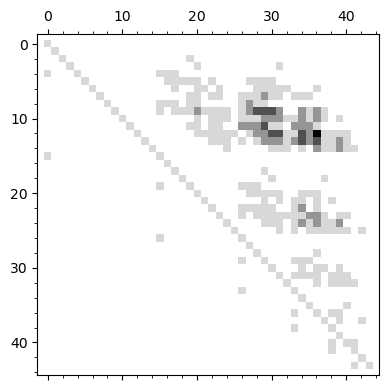}
	\caption{Cartan Matrix of $T_7$ \\ {\small For legibility, the entries are represented as grey values. The entries are integers from 0 (in white) to 4 (the single black pixel).}}
	\label{fig:cartanT7}
\end{figure}
	
	\section*{Conclusion and perspectives}

	The methods presented in this paper provide a new tool for the computational exploration of finite monoids representation theory. We give a method to compute the character table of a finite monoid in the general case as well as a method for the computation of the Cartan matrix. In the latter case, although general algorithms for any finite dimensional algebra already exist, by specializing to monoid algebras, we achieve vastly shorter computation times, thus making the question tractable for bigger monoids. Although we have presented the methods in details only for transformation monoids, the underlying formulas are true in general for finite monoids, and it remains computationally applicable whenever wherever the hypotheses of Section \ref{sec:Algos} are verified. We also invite the reader the consult and test our impementation, available on our github repository\footnote{github.com/ZoltanCoccyx/monoid-character-table}. As this paper is inspired by the combinatorial research on monoid representation theory which have seen renewed activity in recent years we hope that providing this effective tool will allows for the observation of new phenomena.
	
	This work has two natural continuations: improving and expanding.
	For the improvement part, we have noted that by far the most inefficient part of our algorithm is the computation of the radical of the $\lc$-classes, which happens to be the only point where linear algebra is necessary and combinatorics are seemingly not enough. 
	We can ask whether this step could be replaced by a combinatorial computation. 
	Some experiments show that, even in relatively small and very regular cases ($T_5$ for instance) the basis we find for the radical by solving the equation system described in Proposition \ref{prop:equations} does not have easily understandable structure, once the common denominator of the coefficients is eliminated. It therefore seem unlikely to us that a general method for computing the radical of an entirely combinatorial nature exists, although we remain optimistic that in very regular cases (again, $T_n$), the issue lies with us not finding the method rather than it not existing.
	More modestly, in a general context, we could try to exploit further the structure of the equations that define the radical to reduce the size of the system, which is a major bottleneck.

	Another improvement, although perhaps less impactful, could be made by exploiting redundancy: it can happen that two $\lc$-classes $L_1, L_2$ of a submonoid $M$ of $T_n$ are contained in the same $\lc$-class $L$ of $T_n$. Thus, in step $2-b$ of Algorithm \ref{algo:j_class} (for instance), instead of visiting each $\lc$ of $M$, we could visit each $\lc$ of $T_n$ that contain an $\lc$-class of $M$ and count them with some multiplicity. Although this probably would not lead to great improvements in efficiency, this has the advantage of making, in some sense, $T_n$ the worst case scenario, allowing for a finer complexity analysis.
	
	As for extending this work, the natural path seem to adapt these methods for fields of finite characteristic. At this point it appears to us that this question is tractable as the theory remains essentially the same, although it is somewhat difficult to implement in practice. The main hurdle arise, again, when computing the radical of an $\lc$-class: an equivalent of Proposition \ref{prop:rad_is_ne} would have to take into account the role of the radical of the maximal subgroup algebra, which can be non-trivial in positive characteristic. This would translate in needing to effectively compute this radical. Although algorithms are available (for instance in GAP), this is a theoretically difficult and computationally expensive problem, considerably reducing the maximum size of a tractable problem. While modular representation theory is known to be a difficult subject in groups it seems that, again, the situation is not much more complicated for monoids than it is for groups as it is standard practice to reduce monoid theoretic questions to group theoretic ones. Treating modular group representation theory as a black box coming with already existing algorithms (much as we did here for null characteristic group representation theory as a matter of fact), we hope to be able to provide a modular version of our algorithms along with an implementation in the near future.
	
	\section*{Acknowledgements}
	The research work devoted to this project was funded by a PhD grant from the French \emph{Ministère de la recherche et de l'enseignement supérieur}, in the form of a \emph{Contrat Doctoral Spécifique Normalien} attributed for a PhD in the STIC (\emph{Sciences et Technologies de l'Information et de la Communication}) doctoral school of Paris-Saclay University, in the LISN (\emph{Laboratoire Interdisciplinaire des Sciences du Numérique}) under the supervision of Pr. Nicolas Thiéry.
	
	\clearpage
	
	
	\printbibliography
	
\end{document}